\documentclass[11pt]{article}
\usepackage[english]{babel}
\usepackage{fullpage}
\usepackage[utf8]{inputenc}
\usepackage{amsmath,amssymb,amsthm}
\usepackage{mathrsfs}
\usepackage{graphicx}
\usepackage{xcolor}
\usepackage{cite}
\usepackage{float} 
\usepackage{subcaption}
\usepackage{wrapfig}
\usepackage{hyperref}
\hypersetup{
    colorlinks,
    linkcolor={blue!80!black},
    citecolor={blue!80!black},
    urlcolor={blue!80!black}
}
\usepackage{caption}
\usepackage{subcaption}

\newtheorem{theorem}{Theorem}[section]

\newtheorem{proposition}[theorem]{Proposition}

\newtheorem{example}[theorem]{Example}
\newtheorem{remark}[theorem]{Remark}

\def\R{\mathbb{R}}
\def\C{\mathbb{C}}

\def\Cx{\mathbb{C}^{\times}}

\title{Hamiltonian flow between standard module Lagrangians}
\author{Yujin Tong}

\begin{document}
\maketitle

\begin{abstract}
In Aganagic's Fukaya category of the Coulomb branch of quiver gauge theory, the $T_\theta$-brane algebra gives a symplectic realization of the Khovanov--Lauda--Rouquier--Webster (KLRW) algebra,  
where each standard module is known to admit two Lagrangian realizations: the `U'-shaped $T$-brane and the step $I$-brane.  
We show that the latter arises as the infinite-time limit of the Hamiltonian evolution of the former, thus serving as a generalized thimble.  
This provides a geometric realization of the categorical isomorphism previously established through holomorphic disc counting.
\end{abstract}

\section{Introduction}

The Khovanov--Lauda--Rouquier--Webster (KLRW) algebras provide a diagrammatic categorification of tensor product representations of quantum groups and underpin a wide range of advances in higher representation theory and low-dimensional topology \cite{KL1, KL2}. Among their structural features, \emph{standard modules}---introduced by B. Webster---play the role of Verma-like building blocks \cite{Webster1}. Conceptually, they are the smallest quotients of projectives obtained by modding out diagrams with forbidden crossings, and thus form the natural test objects for braid and tangle functors and for dualities in the KLRW framework~\cite{Webster1, Webster2}.
\begin{figure}[H]
\centering
  \includegraphics[width=0.4\textwidth]{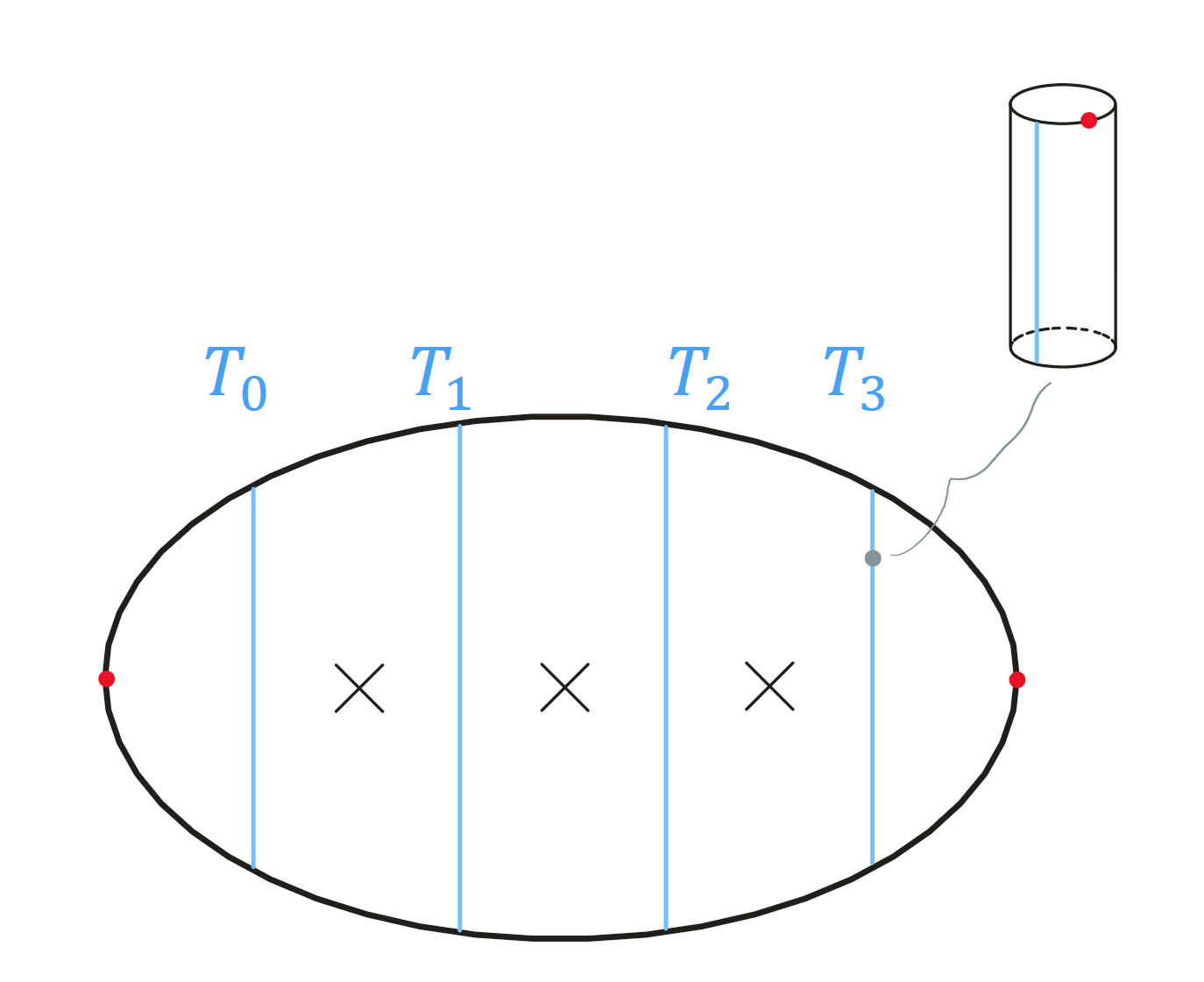}
  \caption{The $T_\theta$-branes representing the generators of the KLRW algebra.}
  \label{0}
\end{figure}

A development due to M. Aganagic is a \emph{symplectic realization} of this algebraic picture \cite{Mina}. In the Fukaya category of the Coulomb branch of quiver gauge theory, 
the \emph{$T_\theta$-brane algebra}, defined by the Lagrangian intersections and holomorphic disc operations, is isomorphic to the KLRW algebra \cite{ADLSZ}; this provides a concrete symplectic model of the diagrammatic generators and their relations.
The branes are illustrated in Fig.~\ref{0}, where the black $\times$ and red dots represent punctures and stops on the manifold, which we will introduce in detail later.

Within this Fukaya-theoretic model, one can ask how algebraically distinguished objects are realized as Lagrangian branes. 
In Aganagic's Fukaya category, the standard modules admit realizations as the Yoneda embeddings of two classes of seemingly very different Lagrangian submanifolds:
\begin{itemize}
    \item the Lagrangian obtained as the symplectic connected sum (equivalently, the mapping cone) of the morphism ${T_{i+1}\to T_i}$, denoted by $U$ in Fig.~\ref{3a}, where we rotate Fig.~\ref{0} by $90^\circ$ for visual clarity;
    \item the \emph{step~$I$--brane} hanging on the $i$th puncture, which we denote by $J$ in Fig.~\ref{3b}. (Here the $T$- and $I$-brane denotes the corresponding shape of the Lagrangians in the fibers.)
\end{itemize}

\begin{figure}[H]
\centering
\begin{subfigure}{0.49\textwidth}
  \centering
  \includegraphics[width=\textwidth]{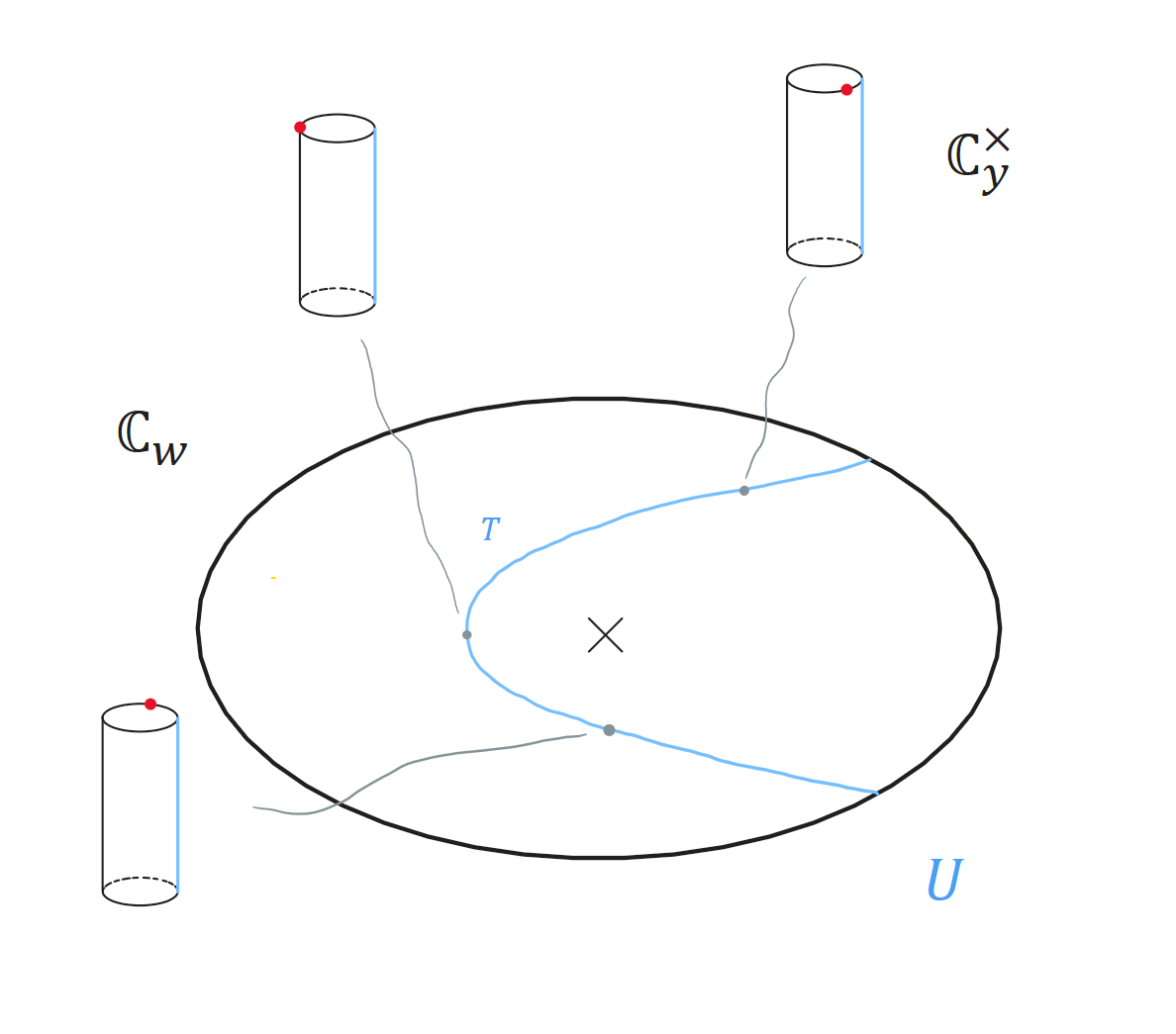}
  \caption{Mapping-cone $T$-brane $U$.}
  \label{3a}
\end{subfigure}
\hfill
\begin{subfigure}{0.49\textwidth}
  \centering
  \includegraphics[width=\textwidth]{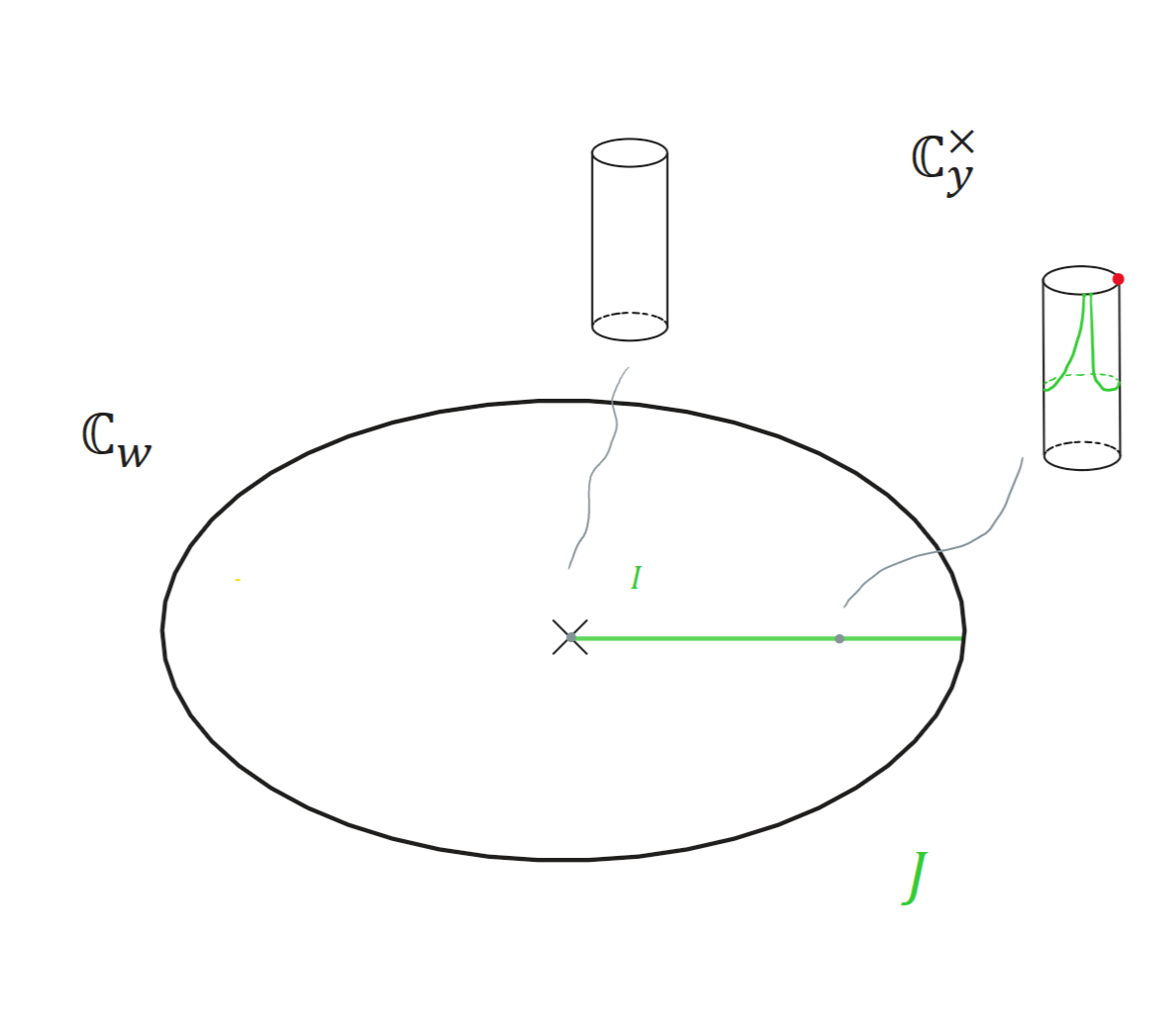}
  \caption{Step $I$-brane $J$.}
  \label{3b}
\end{subfigure}
\caption{Illustration of the standard module Lagrangians.}
\label{3}
\end{figure}

Why do these two geometrically distinct Lagrangians correspond to the same standard module? 
An answer via pseudoholomorphic-disc counting is given in Lemma~10.3 of \cite{Elise}. Namely, there exist unique intersection points 
\[
p\in\mathrm{Hom}(U,J), \qquad q\in\mathrm{Hom}(J,U),
\]
and explicit counting shows that
\[
\mu^2 (p,\, q)=\mathrm{id}\in \mathrm{Hom}(U,U), \qquad \mu^2(q,\, p)=\mathrm{id}\in \mathrm{Hom}(J,J),
\]
proving that the two objects $U$ and $J$ are isomorphic as objects of the Fukaya category
\begin{equation}
  U\cong J.
\end{equation}

While this argument is decisive, it remains intrinsically \emph{categorical}, relying on $A_\infty$-products and holomorphic-disc counts. From the viewpoint of symplectic topology, it is natural to look for a more \emph{geometric explanation}: can one construct a Hamiltonian flow on the Coulomb branch that continuously deforms one Lagrangian into the other, 
thereby providing a geometric interpolation between the two realizations of the standard module?

\begin{theorem}[Main Theorem, informal]
In the Fukaya category of the Coulomb branch associated to the quiver gauge theory $(\Gamma,\vec{d})=(\bullet,1)$, 
the Hamiltonian evolution of the mapping-cone $T$-brane $U$, under flows preserving the fiberwise stops, 
asymptotically converges to the step $I$-brane $J$, 
which serves as the generalized thimble.
\end{theorem}

\paragraph{Outline.}
In Section~2 we review Aganagic's model together with a closely related setup due to Seidel and Smith. 
In Section~3 we give the explicit Hamiltonian function governing the flow. 
Section~4 offers a heuristic description of the deformation process, 
and Section~5 supports the analysis with numerical verification.

\paragraph{Acknowledgements.}
The author would like to thank Peng Zhou and Mina Aganagic for valuable discussions and insightful comments.

\section{Background from the categorification of knot invariants}

In addressing the problem of constructing a Floer-theoretic realization of categorified quantum group link invariants, two main approaches have been proposed.
The first is the symplectic Khovanov homology introduced by P. Seidel and I. Smith \cite{SS}, which provides a geometric model for Khovanov homology via Lagrangian Floer theory.
The second, proposed by M. Aganagic \cite{Mina}, offers a broader framework based on Floer theory with a superpotential.

\subsection{Setup of both models}
In both approaches, the categorification takes the form of the Fukaya category of a horizontal Hilbert scheme.
We recall the setup of the horizontal Hilbert schemes in both theories.

In the Seidel--Smith setup, the horizontal Hilbert scheme is $\mathrm{Hilb}^d_{\mathrm{hor}}(A_{n-1}\to \C)$ where $\pi\colon A_{n-1}\to\C$ is the standard Lefschetz fibration on the Milnor fiber of the $A_{n-1}$-surface singularity \cite{MS}.
More explicitly, the Milnor fiber $A_{n-1}$ is given by
\begin{equation}
  A_{n-1}=\left\{(x,y,w)\in \C^3\middle| xy=\prod_{i=1}^n{(w-a_i)}\right\}
\end{equation}
and $\pi(x,y,w)=w$.

Aganagic's theory, which is inspired by mirror symmetry, has a version for each semisimple Lie algebra and superalgebra.
Let $\Gamma$ be a quiver whose underlying graph is the Dynkin diagram of the Lie algebra, and fix a dimension vector $\vec{d}$.
The considered horizontal Hilbert scheme is the Coulomb branch $\mathcal{M} (\Gamma,\vec{d})$ associated with the corresponding quiver gauge theory $(\Gamma,\vec{d})$.
In addition, one specifies a collection $\mathbf{a}$ of distinguished points in the base, which we will refer to as \textit{punctures}, each labeled by a vertex of the Dynkin diagram. 
There is a natural regular Landau--Ginzburg superpotential function
$\mathcal{W}_{\mathbf{a}} \colon \mathcal{M} (\Gamma,\vec{d}) \longrightarrow \mathbb{C}$ associated with $\mathbf{a}$ introduced in \cite{Mina},
and our category of interest is the Fukaya--Seidel category
$\mathrm{Fuk}(\mathcal{M} (\Gamma,\vec{d}), \mathcal{W}_{\mathbf{a}})$.
\begin{remark}
  For mirror symmetry reasons, Aganagic considered the multiplicative Coulomb branch $\mathcal{M}^{\times} (\Gamma,\vec{d})$ defined in \cite{BFN}. However, in our concerned problem, it suffices to consider the additive version.
\end{remark}

\subsection{The analogue between Aganagic's $\mathfrak{sl}_2$ and Seidel--Smith models}
Aganagic's $\mathfrak{sl}_2$ case where $\Gamma=\bullet$ is relevant for the categorification of Khovanov cohomology. The corresponding horizontal Hilbert scheme is $\mathcal{M} (\bullet,d)=\mathrm{Hilb}^d_{\mathrm{hor}}(\Cx_y\times \C_w\to \C)$ where the projection map $\pi: \Cx_y\times \C_w\to \C$ is given by $\pi(y,w)=w$.
We will refer to $\C_w$ as the base and $\Cx_y$ as the fiber.

The case $d=1$ is of primary interest, because the general theory is largely determined by it \cite{Mina}. When $d=1$, the horizontal Hilbert scheme is simply a manifold $\Cx_y\times\C_w$ and the superpotential is given in (B.3) of \cite{Mina}

\begin{equation}
    \mathcal{W} _{\mathbf{a}}=\lambda _1\log w+\lambda _0\log y+y^{-1}\prod_{i=1}^n{\left( 1-\frac{a_i}{w} \right)},\label{2eq}
\end{equation}
which produces the fiberwise stops. Note that (\ref{2eq}) is defined on the multiplicative Coulomb branch where $w\ne 0$.

The $\log$ terms are introduced for the purpose of equivariant grading, which will not play a role in our discussion. Setting $\hat{y}=\frac{1}{yw^n}\in \Cx$, we rewrite (\ref{2eq}) as
\begin{equation}
  \mathcal{W} _{\mathbf{a}}=\hat{y}\prod_{i=1}^n{\left( w-a_i \right)}.\label{3eq}
\end{equation}
We henceforth write the new variable $\hat{y}$ simply as $y$.

Let us note that in the $\mathfrak{sl}_2$ case, Aganagic's Hilbert scheme is closely related to the Seidel--Smith scheme: Aganagic's differs from Seidel--Smith's only by removing the divisors at the punctures and adding a superpotential:

In the Seidel--Smith setup,
we can see that away from the punctures $\mathbf{a}$, the product on the right-hand side of $ xy=\prod_{i=1}^n{(w-a_i)}$ is nonvanishing, the fiber is biholomorphic to a cylinder $\Cx$. 
At the punctures, the equation degenerates to $xy=0$, thus the fiber collapses to an ordinary double point $\left\{(x=0)\cup (y=0)\right\}$. Hence the Seidel--Smith scheme for the $d=1$ case is locally $\Cx\times\C$ with a divisor added at each punctures $w=a_i$.

In the Aganagic setup, the scheme is globally $\Cx\times\C$, without the double points at punctures. The product in the superpotential (\ref{3eq}) is nonvanishing away from the punctures $\mathbf{a}$, producing a stop on the fiber $\Cx_y$. However, at those punctures, the product is 0, and the stop disappears.

Since later on we will focus only on the local structure near a puncture, we will simply set $a=0$ in the Seidel--Smith model and our additive Coulomb branch. Thus, near a puncture, we can locally regard Seidel--Smith's scheme as $\left\{(x,y,w)\in \C^3\middle| xy=w\right\}$, and Aganagic's scheme as $\Cx_y\times \C_w$ with superpotential $\mathcal{W}=wy$.
These spaces are illustrated in Fig.~\ref{models}.

\begin{figure}[H]
\centering
\begin{subfigure}{0.49\textwidth}
  \centering
  \includegraphics[width=\textwidth]{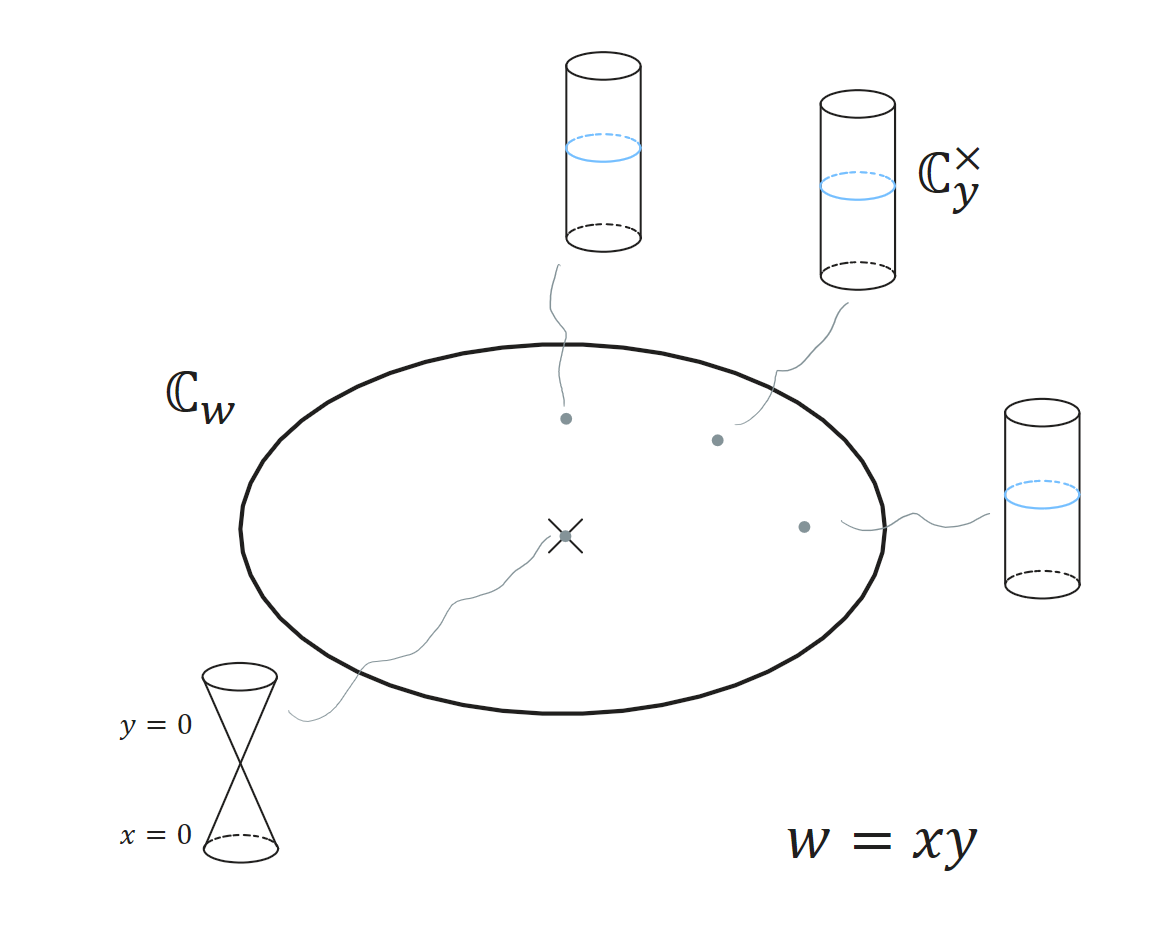}
  \caption{Seidel--Smith model. The blue circles are vanishing cycles along the straight path to the origin.}
  \label{1a}
\end{subfigure}
\hfill
\begin{subfigure}{0.49\textwidth}
  \centering
  \includegraphics[width=\textwidth]{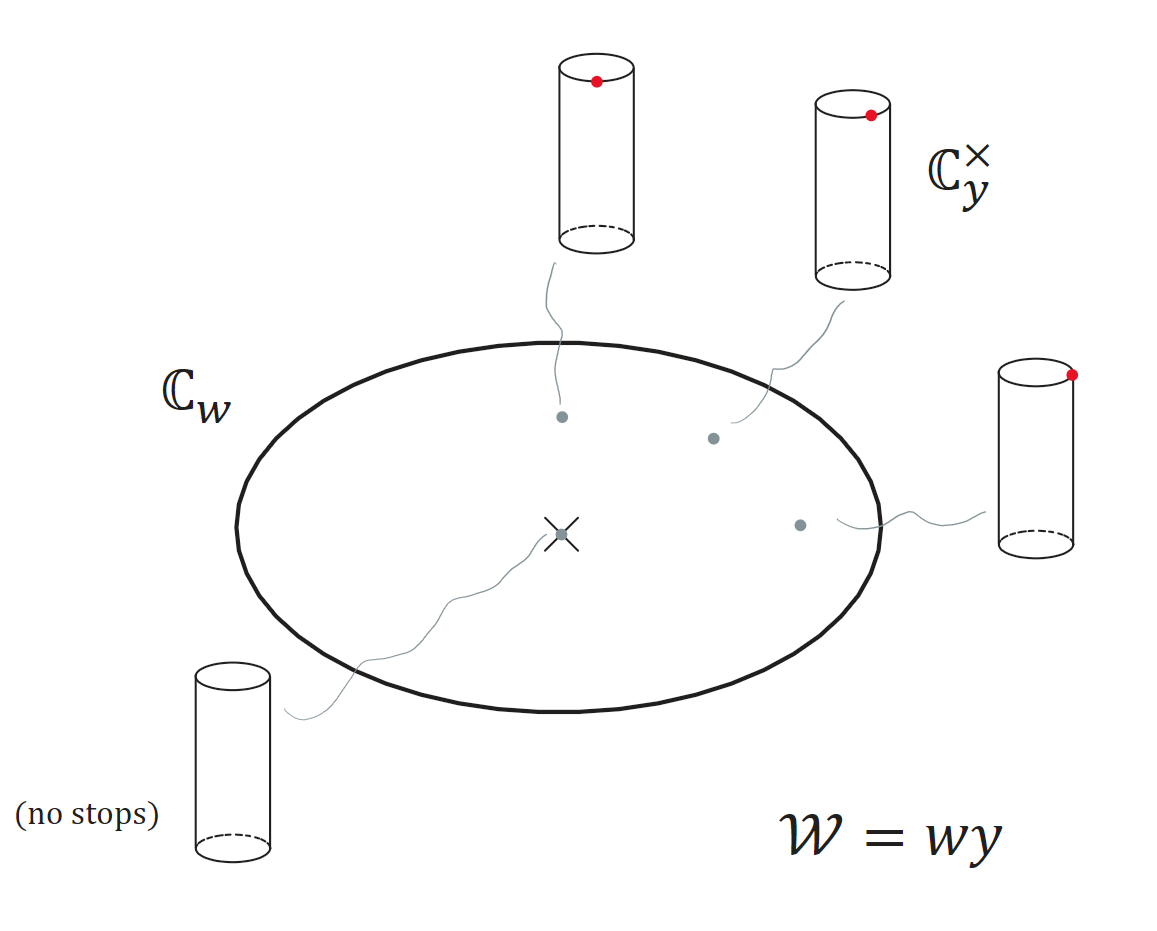}
  \caption{Aganagic $\mathfrak{sl}_2$ model. The red dots are stops determined by superpotential.}
\end{subfigure}
\caption{Illustration of the schemes. $\Cx_y$ is depicted as a cylinder $\R_{\left|y\right|}\times S^1_{\mathrm{Arg}(y)}$ with the top $y\to \infty$ and the bottom $y\to 0$.}\label{models}
\end{figure}

\begin{remark}
  The vanishing cycle over $w\ne 0$ along the straight path to $(x,y,w)=(0,0,0)$ at $w=\epsilon$ is the circle $(x,y)=(\sqrt{\epsilon}e^{it},\sqrt{\epsilon}e^{-it})$ \cite{Thomas}, as shown in blue in Fig.~\ref{1a}.
\end{remark}

\subsection{Geometric proof in the Seidel--Smith analogue}

We can easily see that these two Lagrangian submanifolds $U$ and $J$ are geometrically related in their Seidel--Smith counterpart, where the analogue for the $T$-fiber in the Seidel--Smith setting is still the positive real line $\R_+ \subset \Cx$, while the $I$-fiber corresponds to a circle $U(1) \subset \Cx$. We will name the Lagrangians analogous to $U$ and $J$ in the Seidel--Smith setting respectively as $\tilde{U}$ and $\tilde{J}$, as shown in Fig.~\ref{13}.

\begin{proposition}
Under the Hamiltonian flow $\varphi_t$ generated by $H=\mathrm{Im}(w)$, the Lagrangian $\tilde U$ converges, as $t\to+\infty$, to the Lagrangian $\tilde J$.
\end{proposition}

\begin{proof}
On a Kähler manifold, the Hamiltonian vector field of $H=\mathrm{Im}(w)$ coincides with the gradient of $\mathrm{Re}(w)$.
Consequently, $\mathrm{Im}(w)$ is conserved along the flow, while $\mathrm{Re}(w)$ increases monotonically.
Thus the flow on the base runs parallel to the real axis, escaping towards $+\infty$.
In the fibers, one has vanishing cycles, whose transport along vanishing paths gives rise to the associated Lefschetz thimbles.
From this point of view, the Lagrangian $\tilde J$ can be identified with the unstable Lefschetz thimble of the critical point $(y,z,w)=(0,0,0)$.

On the other hand, the Lagrangian $\tilde U$ intersects the stable manifold only at a single point $p$ where $(x,y,w)=(\sqrt{|w^*|},\sqrt{|w^*|},w^*)$.
Under the flow, every point except a neighborhood of $p$ escapes to $+\infty$, while the neighborhood of $p$ is carried into the unstable manifold.  
This unstable manifold is precisely the Lagrangian $\tilde J$.  
Thus, under the long-time Hamiltonian dynamics of $H$, $\tilde U$ converges to $\tilde J$.
\end{proof}

\begin{figure}[H]
  \centering
  \includegraphics[width=.6\textwidth]{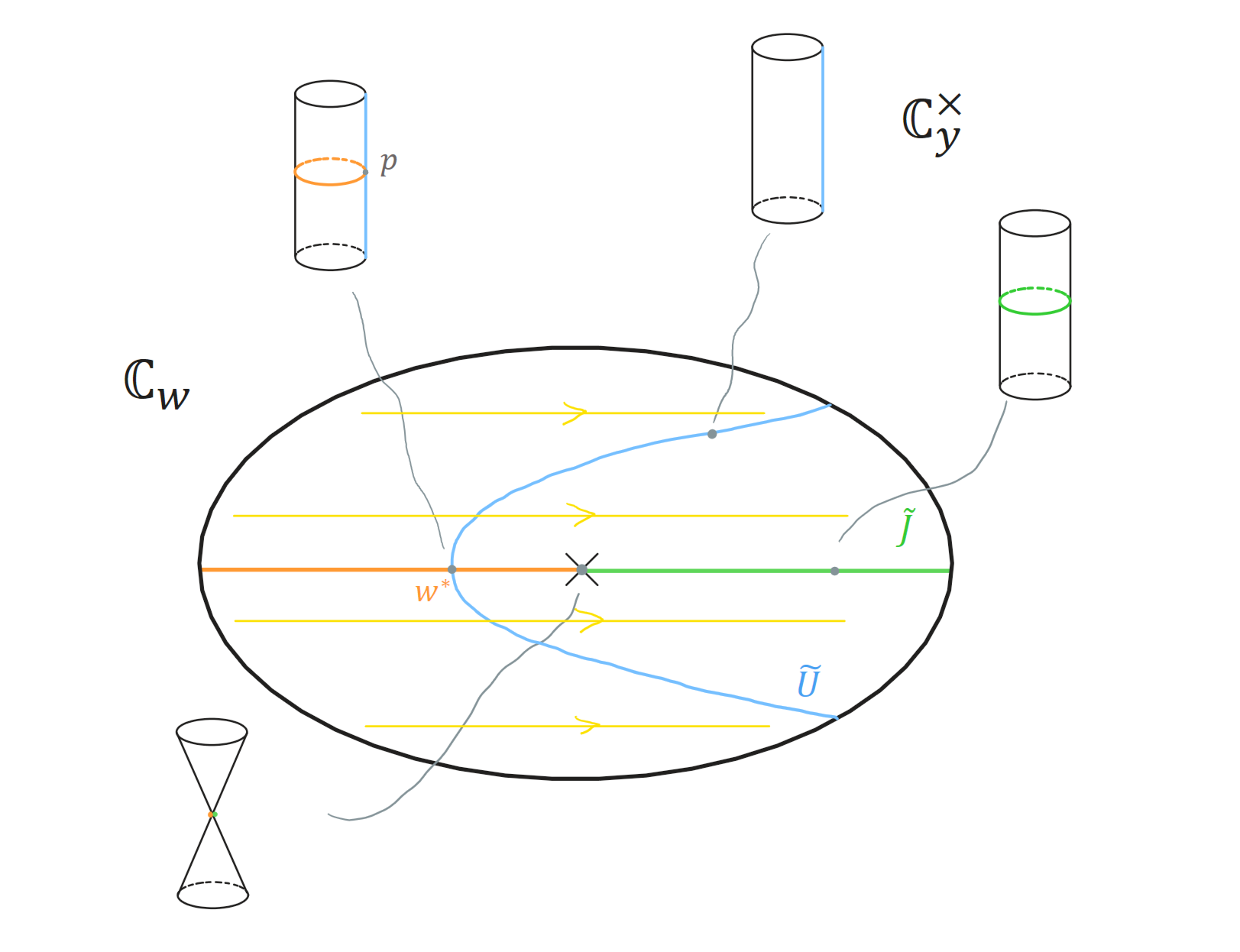}
  \caption{The Lagrangian $\tilde{U}$ (blue), $\tilde{J}$ (green, coinciding with the unstable manifold), and the stable manifold (orange).}
  \label{13}
\end{figure}

\begin{remark}
  More precisely, this kind of `Hamiltonian isotopy in the $t\to \infty$ limit' should be described as $\tilde U$ converges locally uniformly in the Gromov--Hausdorff sense to $\tilde J$, which means that for any compact set $K$, $\lim_{t\to \infty}\mathrm{dist}(\varphi_t(\tilde U \cap K),\tilde J\cap K)=0$.
\end{remark}

\section{Main result: Geometric proof in Aganagic's model}

In Aganagic's model, we can no longer apply the thimble argument, which renders the geometric equivalence nontrivial.

The main result of this paper establishes the existence of a Hamiltonian flow under which the Lagrangian $J$ arises as the $t\to \infty$ limit of the Lagrangian $U$.

\begin{theorem}
In Aganagic's $\mathfrak{sl}_2$, $d=1$ model, consider a neighborhood of a puncture with the space $\C_w \times \Cx_y$ and superpotential $\mathcal{W}=wy$. Equip this space with the symplectic form
\begin{equation}
  \omega = \frac{i}{2}\left( \mathrm{d}w\wedge \mathrm{d}\bar{w} + |y|^{-3}\,\mathrm{d}y\wedge \mathrm{d}\bar{y} \right). \label{omega}
\end{equation}
Then the Hamiltonian
\begin{equation}
  H(w,y) = \Bigg( \frac{1}{2}e^{-2/|y|}\left( 1-\frac{e^{-|w|^2}}{|w|}\int_0^{|w|} e^{\alpha^2}\,\mathrm{d}\alpha \right) 
  + 1 - e^{-1/|y|} \Bigg)\,\mathrm{Im}(w)
\end{equation}
generates a flow such that the Lagrangian $U$ converges to the Lagrangian $J$ as $t\to\infty$.
\end{theorem}

Details of the construction and properties of this Hamiltonian function will be provided in the following sections.

\newpage
\section{Heuristic description of the flow}
\subsection{Structure and coordinates of the manifold}

\begin{figure}[h!]
  \centering
  \includegraphics[width=.8\textwidth]{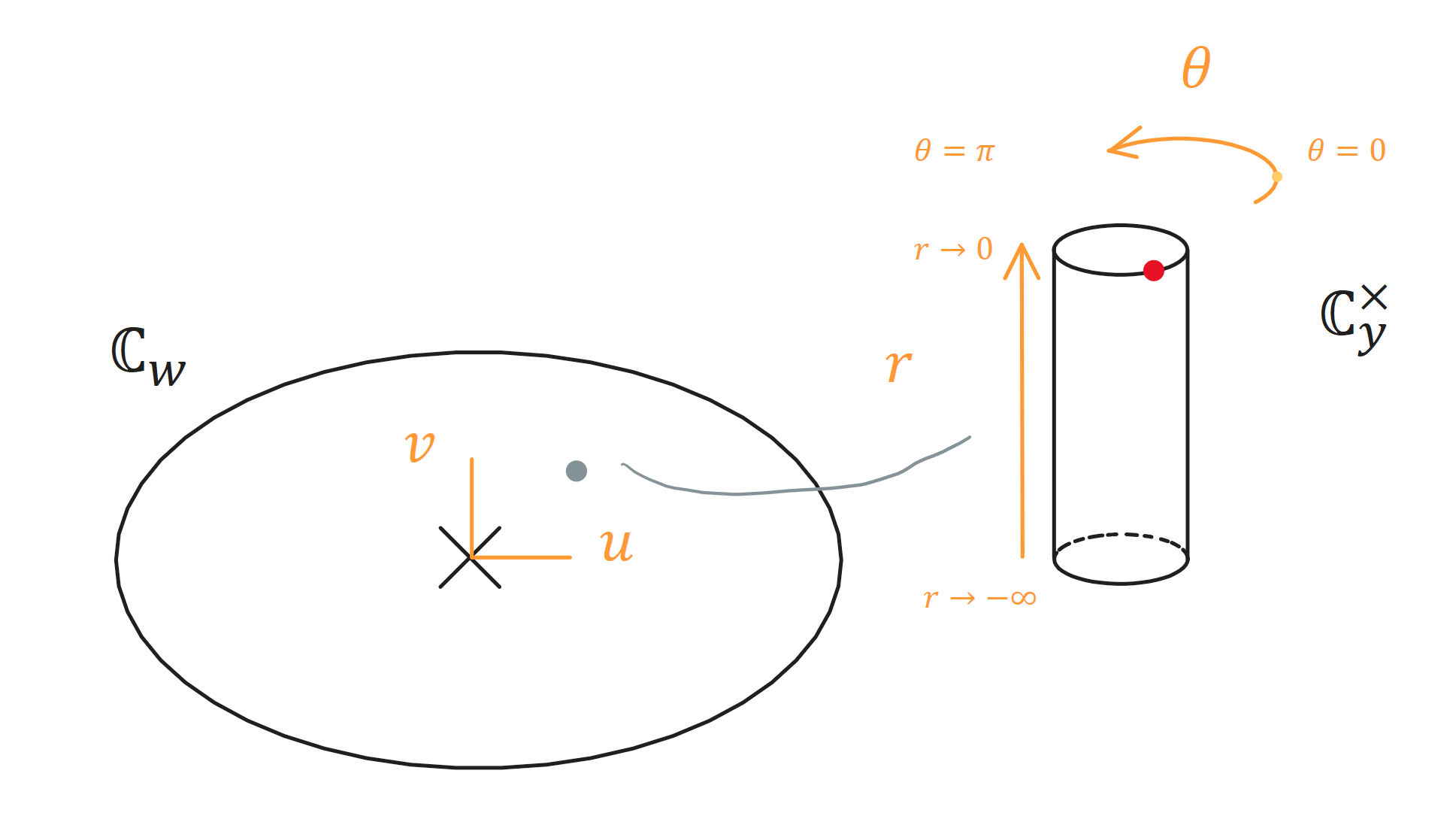}
  \caption{Real coordinates for Aganagic's scheme.}
  \label{1}
\end{figure}

Following Fig.~\ref{1}, we introduce real coordinates $(u,v,r,\theta)\in \R\times\R\times\left(-\infty,0\right)\times S^1$
\begin{gather}
  w=u+iv,\quad y=-\frac{e^{i\theta}}{r}
\end{gather}
where $\theta\sim\theta+2\pi$ is the cyclic coordinate.
We will also use polar coordinates for the base,
\begin{equation}
  w=Re^{i\Theta}
\end{equation}when convenient.

In these coordinates, the above symplectic structure (\ref{omega}) becomes
\begin{equation}
  \omega=\mathrm{d}u\wedge\mathrm{d}v+\mathrm{d}r\wedge\mathrm{d}\theta=R\,\mathrm{d}R\wedge\mathrm{d}\Theta+\mathrm{d}r\wedge\mathrm{d}\theta
\end{equation}

The stops are located where $\mathcal{W}=wy\to +\infty$, corresponding to the limit $r\to 0$ and $\theta=-\Theta$. The Hamiltonian is required to preserve these stops.

\subsection{Expected behavior of the flow}

Intuitively, the flow should move in the positive $u$-direction and be symmetric with respect to the $u$-axis.  
For simplicity, we assume that the Hamiltonian $H$ is independent of $\theta$, which implies that $r$ is conserved along flow lines.  

At $r\to 0$, since there is no stop at the fiber over $w=0$ and the flow must preserve the stops, the Hamiltonian vector field must vanish at $w=0$.  
For points on the negative real axis ($w\in \R_-$), it takes infinite time to reach $w=0$.
For points with nonzero imaginary part ($v=\mathrm{Im}(w)\ne 0$), trajectories can cross the $v$-axis and eventually escape to $w=+\infty$.  

\begin{wrapfigure}{r}{0.45\textwidth}
  \centering
  \includegraphics[width=0.4\textwidth]{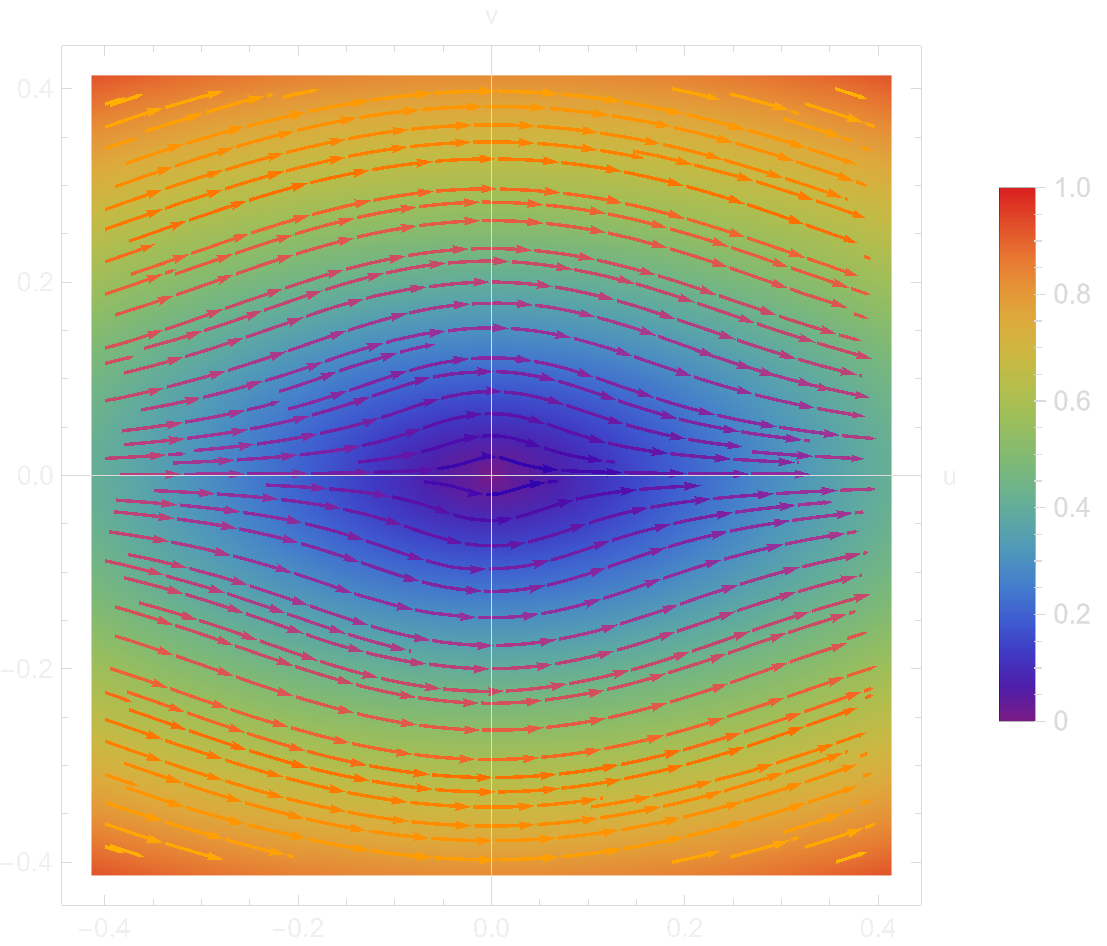}
  \caption{Flow of $H=Rv$. Colors represent speed.}
  \label{4}
\end{wrapfigure}

\begin{example}
  A Hamiltonian of the form $H=Rv=\sqrt{u^2+v^2}\,v$ exhibits this behavior (see Fig.~\ref{4}).  
  More generally, functions of the form $H=f(R)v$ with suitable $f(R)$ produce a similar effect.
\end{example}

For the fibers, because the flow must preserve the stops, which rotate in opposite directions across different fibers, there should be corresponding rotations within the fibers, which will be discussed later.  

At $r<0$ and $v=0$, the stops are no longer relevant.  
Hence points in the fibers over the negative $u$-axis should flow past the origin and onto the positive $u$-axis.  
By symmetry, there should be no fiber rotations in this region.  

As $r\to -\infty$, from a global perspective, every point lies close to the $u$-axis rather than near the tops of the cylindrical fibers.  
Thus, the flow should behave as along the $u$-axis for $r<0$: a simple rightward translation, without fiber motion.  

\begin{example}
  A Hamiltonian of the form $H=\mathrm{Im}(w)=v$ produces exactly this behavior.
\end{example}

\begin{figure}[H]
\centering
\begin{subfigure}{0.45\textwidth}
  \centering
  \includegraphics[width=\textwidth]{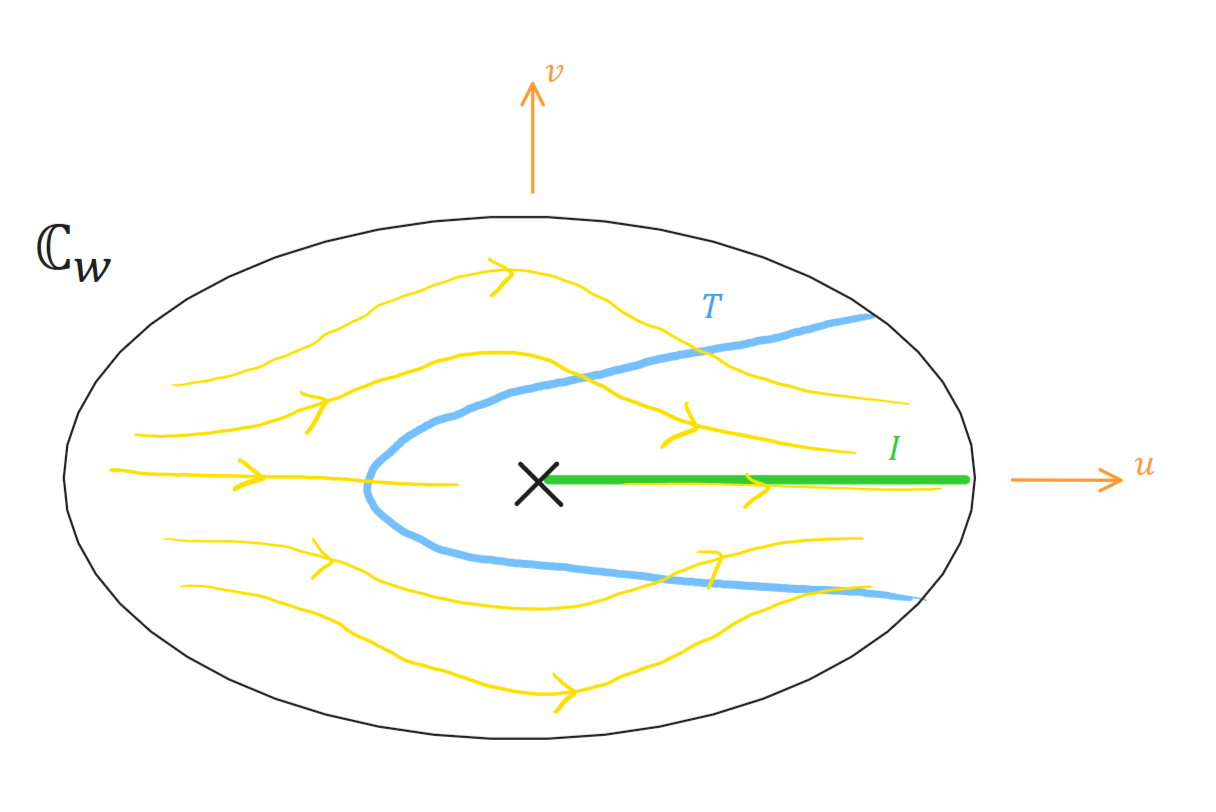}
  \caption{$r\to 0$.}
\end{subfigure}
\hfill
\begin{subfigure}{0.45\textwidth}
  \centering
  \includegraphics[width=\textwidth]{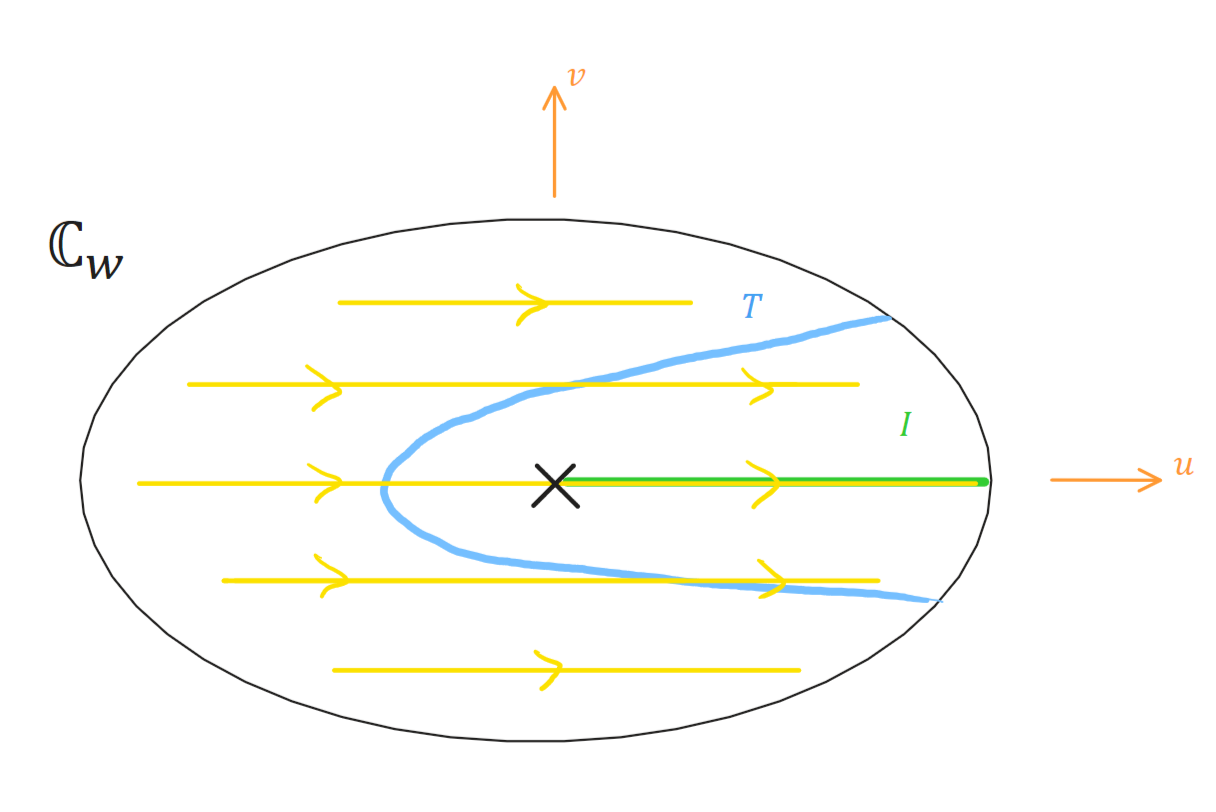}
  \caption{$r\to -\infty$.}
\end{subfigure}
  \caption{Expected behavior of the flow in the base.}
  \label{52}
\end{figure}
~\\
~\\
\subsection{Emergence of the $I$-fibers from the $T$-fibers}
We will see that any flow with the above properties necessarily carries the Lagrangian $U$ to the Lagrangian $J$ in the limit.

The entire Lagrangian $J$ is expected to originate from a small neighborhood around $(u,v,r,\theta)=(-u^*,0,0,0)$ on $U$, where $-u^*$ is the real part of the intersection point of $U$ with the $u$-axis. Specifically, points on the $J$ within a given fiber are expected to emerge from a small semicircle in $U$ centered at $(u,v,r,\theta)=(-u^*,0,0,0)$.

\begin{figure}[H]
  \centering
  \includegraphics[width=\textwidth]{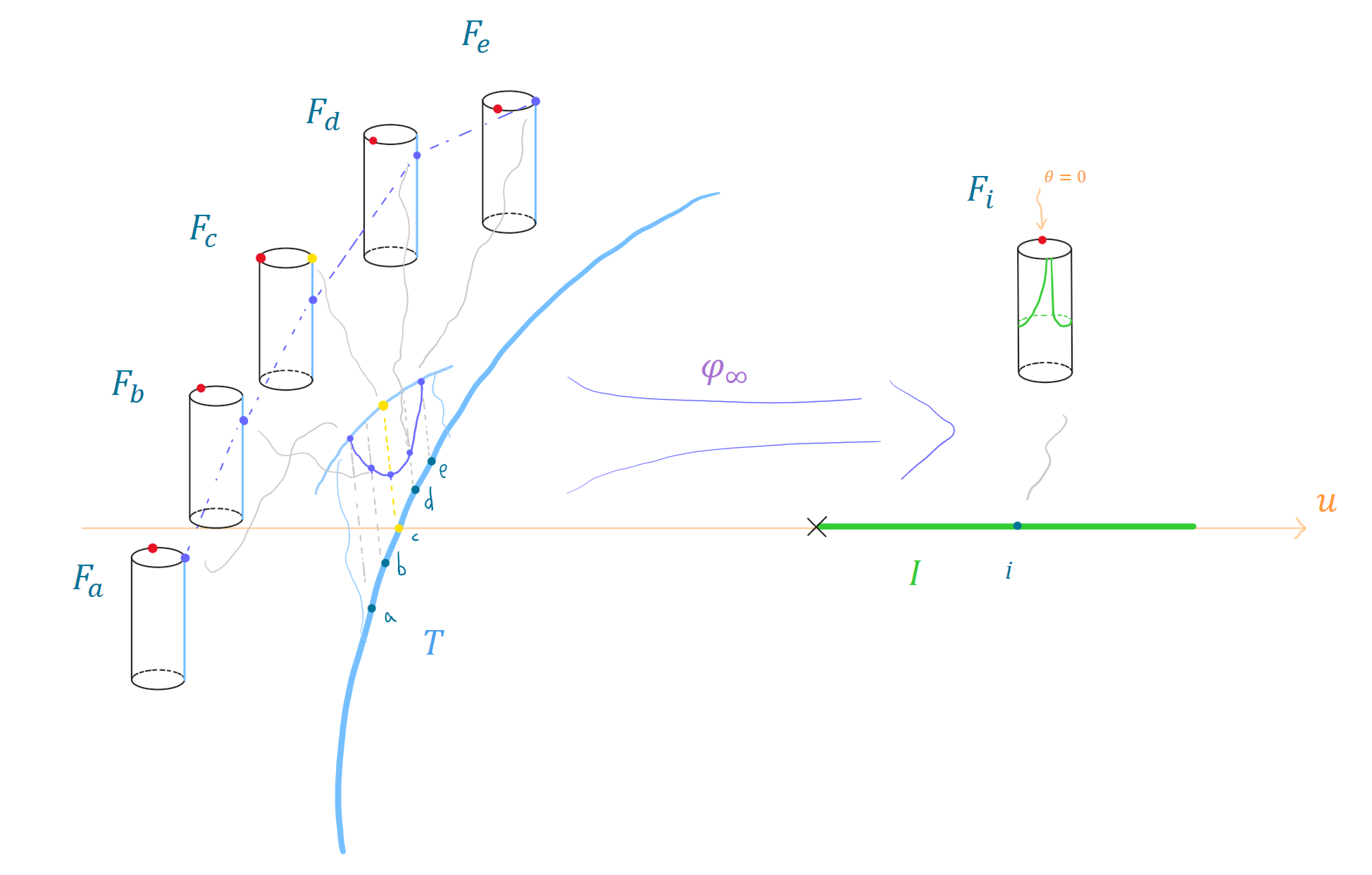}
  \caption{The preimage of an I-fiber under the flow of $H$. The fiber $F_i$ is shown tilted for visualization purposes.}
  \label{6}
\end{figure}

For points near the top of the $\C^*_y$ fiber ($r=0$, i.e., the part of the semicircle far from the $-u$-axis), the fiber points are expected to rotate along with the stops.

For semicircle points closer to the $-u$-axis, the rotation decreases as they move across fibers. In particular, points exactly on the $-u$ axis simply translate rightward without any rotation or deviation.

Along the flow, semicircle points with $v\ne 0$ acquire a nonzero $\theta$ component as their $v$ coordinate approaches zero. In the limit $t\to\infty$, the semicircle collapses onto a single fiber on the positive $u$-axis, which constitutes an I-fiber. For a detailed illustration, see Fig.~\ref{7}.

\begin{figure}[H]
  \centering
  \includegraphics[width=\textwidth]{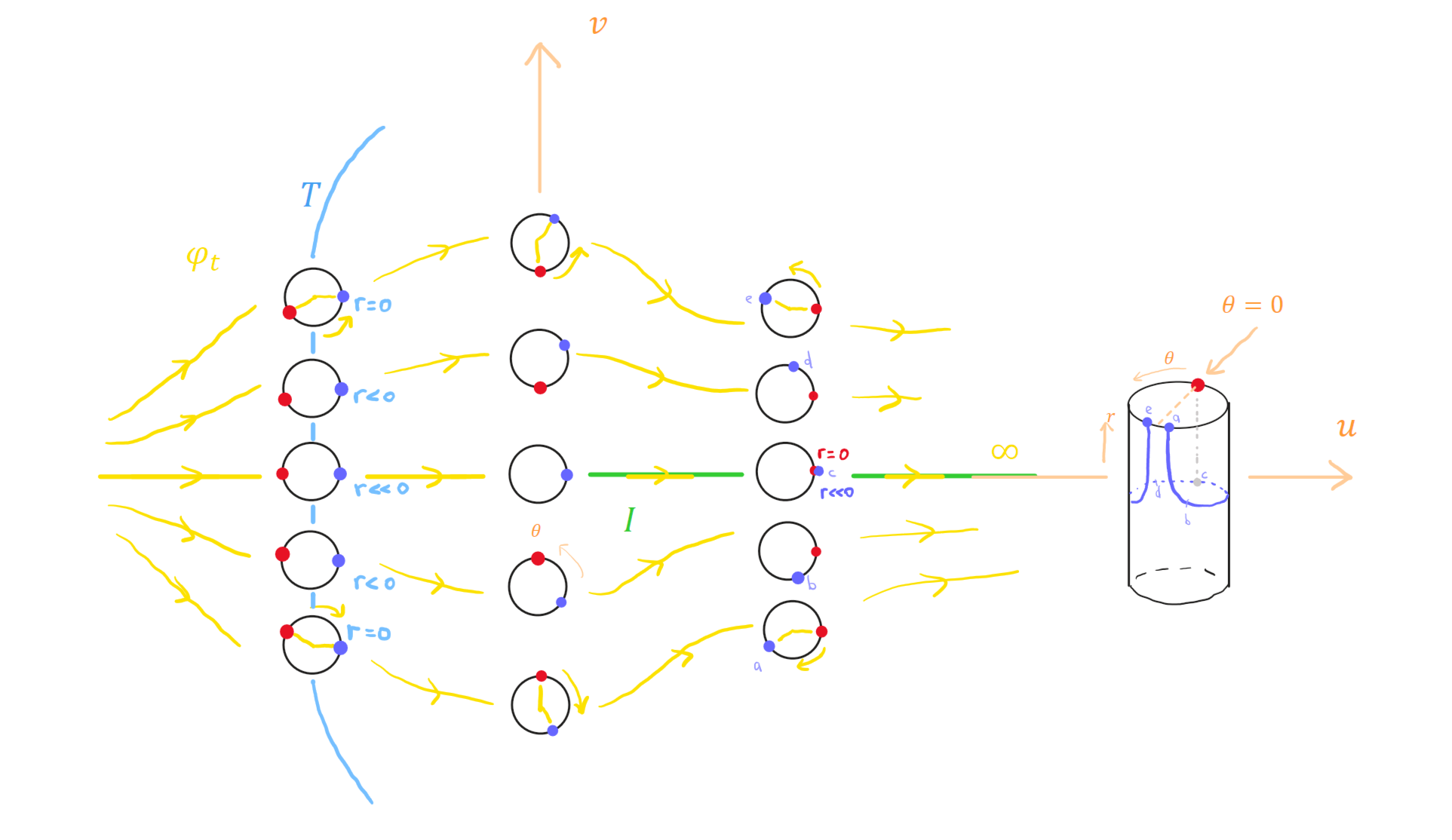}
  \caption{The mechanism of the Hamiltonian flow of $H$. The circles provide a top-down view of the fibers. The red dots indicate stops located at the top ($r=0$) of the cylindrical fiber. The purple dots correspond to points in the semicircle under the Hamiltonian flow, lying at different heights ($r$); these heights remain fixed along the flow.}
  \label{7}
\end{figure}

\begin{remark}
  It should be noted from Fig.~\ref{7} that the notation $r\ll 0$ for the purple dot on the fiber with $v=0$ is only meant to indicate that its $r$-value is smaller than those of the neighboring $r<0$ points on the semicircle. In the limit $t\to\infty$, this point approaches $r=0$, as do all the purple dots. Nevertheless, the point in the $v=0$ fiber remains slightly lower than the other points in the semicircle, by a higher-order infinitesimal. This distinction is important because of the stop on the fibers along the positive $u$-axis.
\end{remark}

\section{Construction of the Hamiltonian function}
The question remains whether we can really write out a Hamiltonian function to achieve such flow behavior.

We propose the ansatz:
\begin{equation}
  H\left( R,\Theta ,r \right) =\left( a\left( r \right) f\left( R \right) R+b\left( r \right) R \right)  \sin \Theta, \label{a}
\end{equation}
where the functions $a(r)$ and $b(r)$ satisfy the boundary conditions:
\begin{gather}
  a(0) =1, \quad b(0) =0,  \\
  a(-\infty) =0, \quad b(-\infty) =1.
\end{gather}

The Hamiltonian vector field produced by a general Hamiltonian $H(R,\Theta,r,\theta)$ would be
\begin{equation}
X_H=\frac{1}{R}\frac{\partial H}{\partial \Theta}\,\partial_R
-\frac{1}{R}\frac{\partial H}{\partial R}\,\partial_\Theta
+\frac{\partial H}{\partial \theta}\,\partial_r
-\frac{\partial H}{\partial r}\,\partial_\theta.
\end{equation}

At $r=0$, the fibers must rotate with angular velocity equal in magnitude and opposite in sign to that of the base. This requirement translates into the condition
\begin{equation}
  \frac{\mathrm{d}\Theta}{\mathrm{d}t}+\frac{\mathrm{d}\theta}{\mathrm{d}t}=0 \quad \text{at } r=0,\; R\ne 0,
\end{equation}
which is
\begin{equation}
  \frac{1}{R} \frac{\partial H}{\partial R} + \frac{\partial H}{\partial r} = 0 \quad \text{at } r=0,\; R\ne 0.
\end{equation}
Substituting the ansatz (\ref{a}) yields the differential equation
\begin{gather}
  \left( \frac{1}{R}f\left( R \right) +f'\left( R \right) +\left( a'\left( 0 \right) f\left( R \right) +b'\left( 0 \right) \right) R \right) \sin \Theta =0.
\end{gather}
Since this identity must hold for all $\Theta$, the coefficient of $\sin\Theta$ must vanish. Therefore,
\begin{gather}
  f'\left( R \right) +\left( \frac{1}{R} + a'\left( 0 \right) R \right) f\left( R \right) +b'\left( 0 \right) R =0,\label{b}
\end{gather}
where we have already incorporated the conditions $a(0)=1$ and $b(0)=0$.

The general solution to (\ref{b}) is given by
\begin{equation}
  f\left( R \right) R=-\frac{b'\left( 0 \right)}{a'\left( 0 \right)}R+\sqrt{\frac{\pi}{2}}\frac{b'\left( 0 \right) e^{-a'\left( 0 \right) R^2/2}}{a'\left( 0 \right) ^{3/2}}\operatorname{erfi}\left( \sqrt{\frac{a'\left( 0 \right)}{2}}R \right) +Ce^{-a'\left( 0 \right) R^2/2},
\end{equation}
where $C\in\R$ is an arbitrary constant.
Here, the imaginary error function is defined as
\begin{equation}
  \operatorname{erfi}(x) = \frac{2}{\sqrt{\pi}} \int_0^x e^{\alpha^2} \,\mathrm{d}\alpha.
\end{equation}

To ensure that the Hamiltonian is well-defined, $f(R)R$ must be zero at $R=0$. This imposes the condition:
\begin{gather}
  C=0.
\end{gather}

Thus, the Hamiltonian takes the form:
\begin{align}
  H\left( R,\Theta ,r \right) &=\left( a\left( r \right) f\left( R \right) R+b\left( r \right) R \right) \sin \Theta\notag  \\
  &=\left( a\left( r \right) \left( -\frac{b'\left( 0 \right)}{a'\left( 0 \right)}+\sqrt{\frac{\pi}{2}}\frac{b'\left( 0 \right) e^{-a'\left( 0 \right) R^2/2}}{a'\left( 0 \right) ^{3/2}}\mathrm{erfi}\left( \sqrt{\frac{a'\left( 0 \right)}{2}}R \right) \right) R+b\left( r \right) R \right) \sin \Theta .
\end{align}

To achieve the desired flow behavior, we choose:
\begin{gather}
  a\left( r \right) =e^{2r}, \quad b\left( r \right) =1-e^r,  \\
  a'\left( 0 \right) =2, \quad b'\left( 0 \right) =-1.
\end{gather}

Substituting these into our Hamiltonian, we obtain
\begin{gather}
  H\left( R,\Theta ,r \right) =\left( \frac{1}{2}e^{2r}\left( 1-\frac{e^{-R^2}}{R}
  \int_0^R e^{\alpha^2} \,\mathrm{d}\alpha \right) +1-e^r \right) R \sin \Theta.
\end{gather}

Equivalently, in Cartesian coordinates $(u,v,r)$,
\begin{equation}
  H\left( u,v,r \right) =\left( \frac{1}{2}e^{2r} \left( 1-\frac{e^{-(u^2+v^2)}}{\sqrt{u^2+v^2}}
  \int_0^{\sqrt{u^2+v^2}} e^{\alpha^2} \,\mathrm{d}\alpha \right) +1-e^r \right) v,
\end{equation}
or in complex coordinates $(w,y)$,
\begin{equation}
  H\left( w,y \right) =\left( \frac{1}{2}e^{-2/\left| y \right|}\left( 1-\frac{e^{-\left| w \right|^2}}{\left| w \right|}\int_0^{\left| w \right|}{e^{\alpha ^2}\mathrm{d}\alpha} \right) +1-e^{-1/\left| y \right|} \right) \mathrm{Im}\left( w \right) .\label{27}
\end{equation}

\section{Verification of the Hamiltonian flow}

The Hamiltonian vector field of (\ref{27}) is, by a direct computation
\begin{align}
X_H=&\frac{1}{2}\left( \left( 1-e^r \right) \left( R^2+u^2 \right) +e^{2r}u^2+e^{2r}\left( -u^2+2R^2v^2 \right) \frac{e^{-R^2}}{R}\int_0^R{e^{\alpha ^2}\mathrm{d}\alpha} \right) \frac{1}{R^2}\;\partial _u\notag
\\
&+\frac{1}{2}e^{2r}\left( 1-\left( 1+2R^2 \right) \frac{e^{-R^2}}{R}\int_0^R{e^{\alpha ^2}\mathrm{d}\alpha} \right) \frac{uv}{R^2}\;\partial _v\notag
\\
&+\left( e^r-e^{2r}\left( 1-\frac{e^{-R^2}}{R}\int_0^R{e^{\alpha ^2}\mathrm{d}\alpha} \right) \right) v\;\partial _{\theta}
\end{align}

We can directly check this vector field is well defined. Especially, at $w\to 0$,
\begin{equation}
  \lim_{(u,v)\to(0,0)}{X_H(u,v,r)}=(1-e^r)\partial_u.
\end{equation}
which is expected. We can also see that the flow in the base, shown in Fig.~\ref{MMA1} does obey the descriptions in Fig.~\ref{52}.

We conclude this paper by presenting numerical computations and plots that directly verify that our Hamiltonian realizes the equivalence of the two Lagrangians. Since it is impossible to visualize a four-dimensional flow in a figure, we project the dynamics onto lower-dimensional subspaces by suppressing some of the parameters.

The only difference between this flow and the expected behavior described above is that, over finite time, the angles of rotation for the points $b$ and $d$ in Fig.~\ref{7} may be larger than those for the points $a$ and $e$ at the top, as shown in Figs.~\ref{MMA2} and \ref{MMA3}. However, this discrepancy is resolved as $t\to\infty$.
The only difference of this flow from the expected behavior we described above is that, under finite time, the angle that the points $b$ and $d$ in Fig.~\ref{7} rotate may be larger than $a$ and $e$ at the top, as shown in Figs.~\ref{MMA2} and \ref{MMA3}. But this is not a problem since it will be resolved as $t\to \infty$.

\begin{figure}[H]
\centering
\begin{subfigure}{0.49\textwidth}
  \centering
  \includegraphics[width=\textwidth]{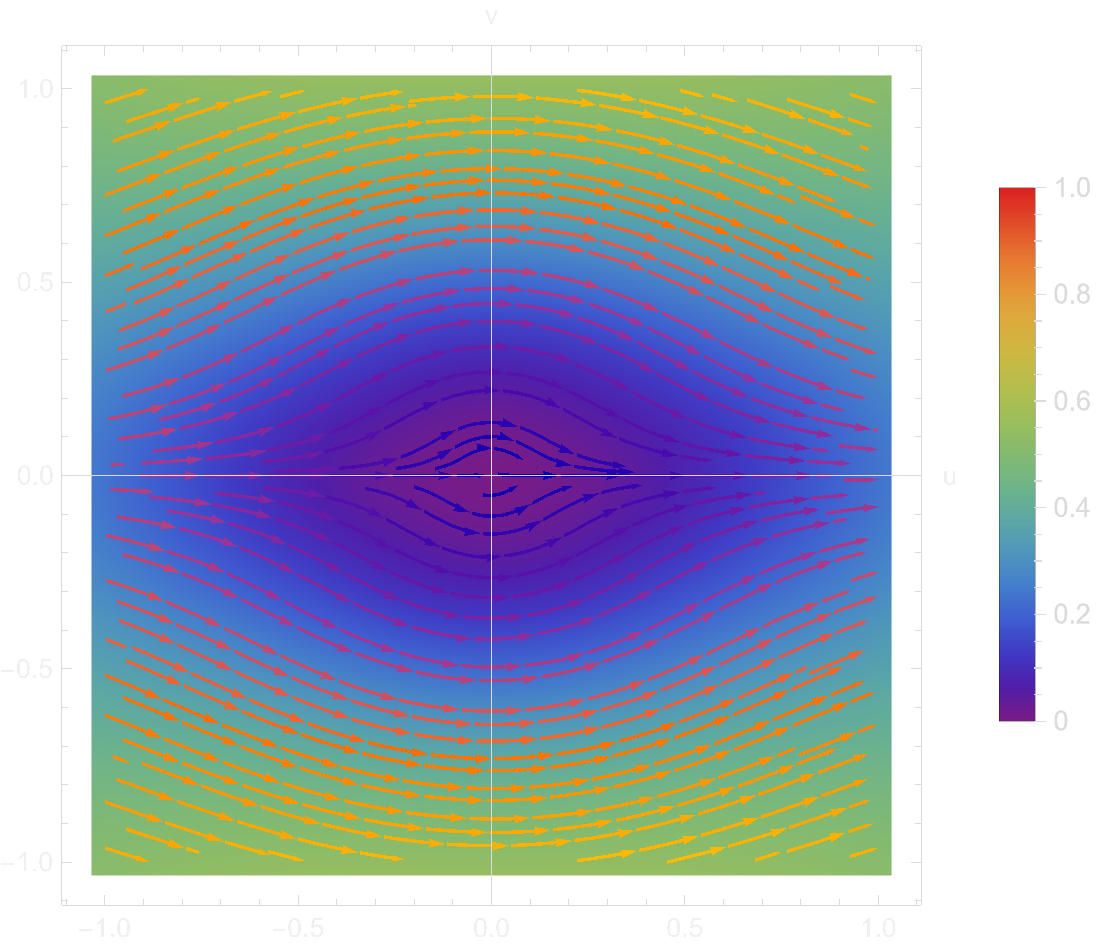}
  \caption{$r=0$.}
\end{subfigure}
\hfill
\begin{subfigure}{0.49\textwidth}
  \centering
  \includegraphics[width=\textwidth]{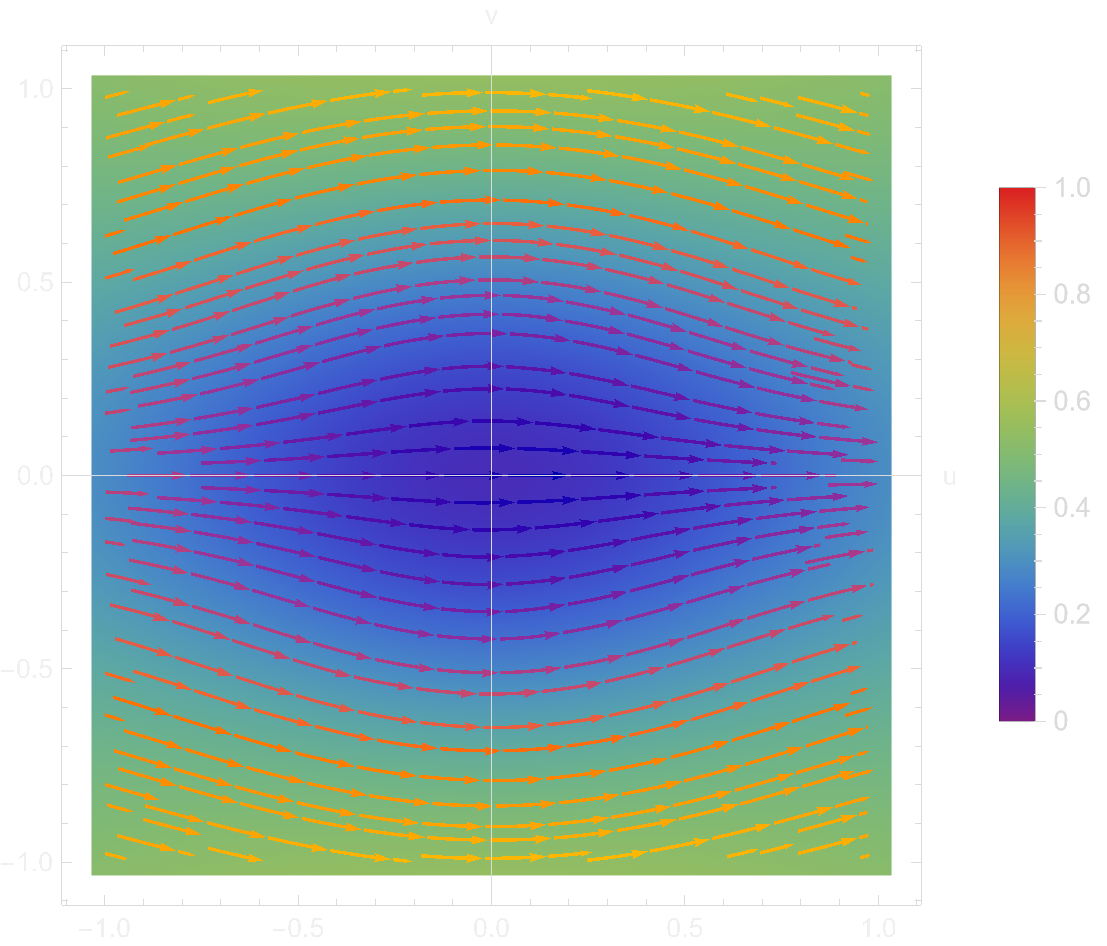}
  \caption{$r=-0.1$.}
\end{subfigure}\\
\begin{subfigure}{0.49\textwidth}
  \centering
  \includegraphics[width=\textwidth]{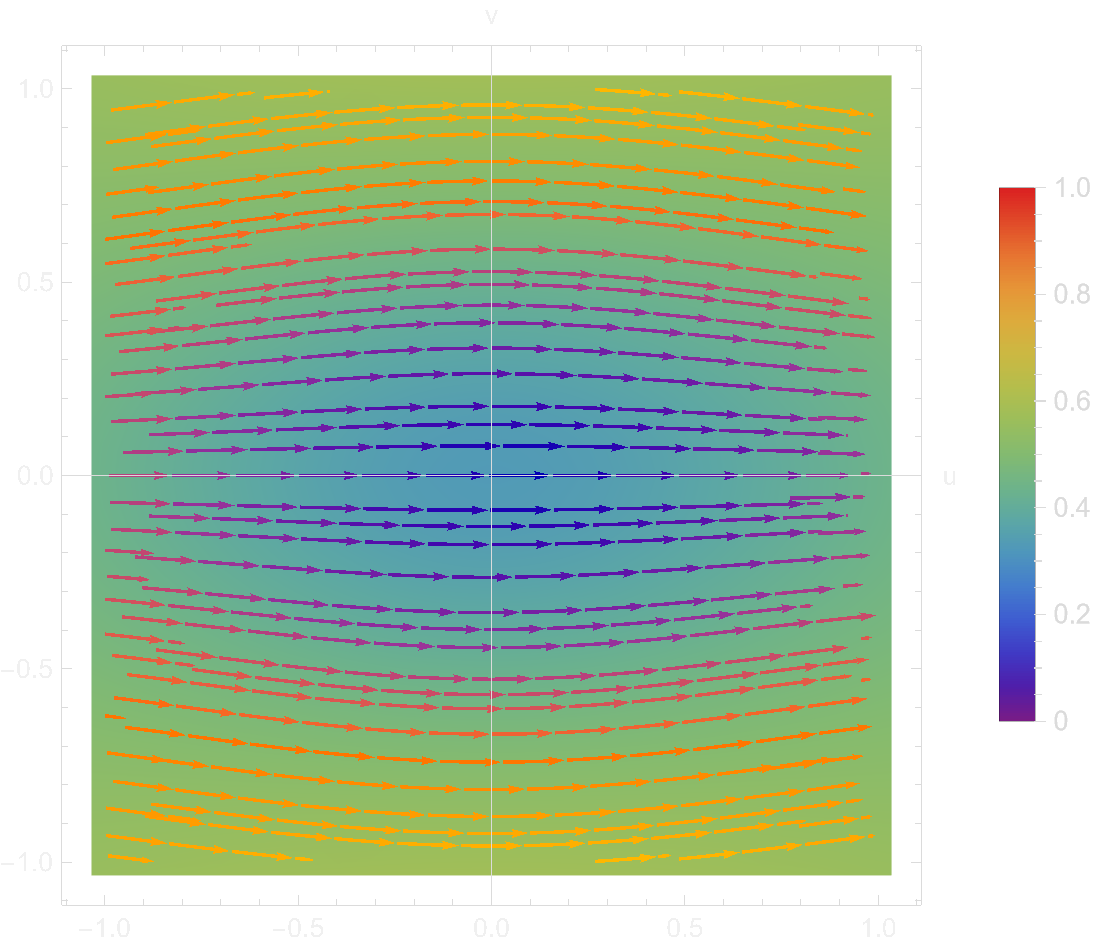}
  \caption{$r=-0.4$.}
\end{subfigure}
\hfill
\begin{subfigure}{0.49\textwidth}
  \centering
  \includegraphics[width=\textwidth]{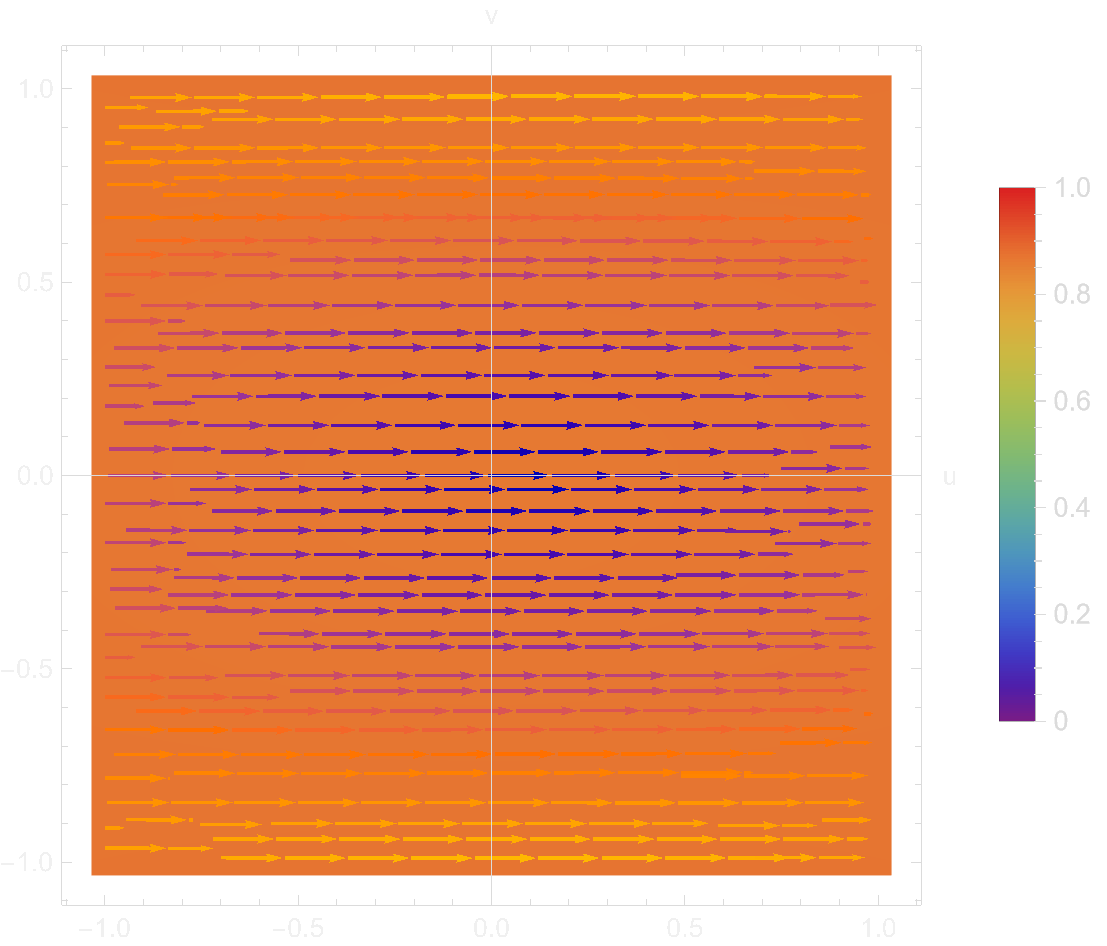}
  \caption{$r=-2$.}
\end{subfigure}
\caption{Flow in the base $\C_w$ at various $r$. Colors represent the speed of the flow.}\label{MMA1}
\end{figure}

\begin{figure}[H]
\centering
\begin{subfigure}{0.47\textwidth}
  \centering
  \includegraphics[width=\textwidth]{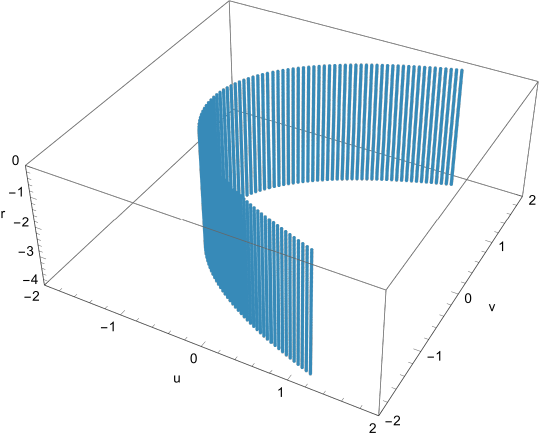}
  \caption{$t=0$.}
\end{subfigure}
\hfill
\begin{subfigure}{0.47\textwidth}
  \centering
  \includegraphics[width=\textwidth]{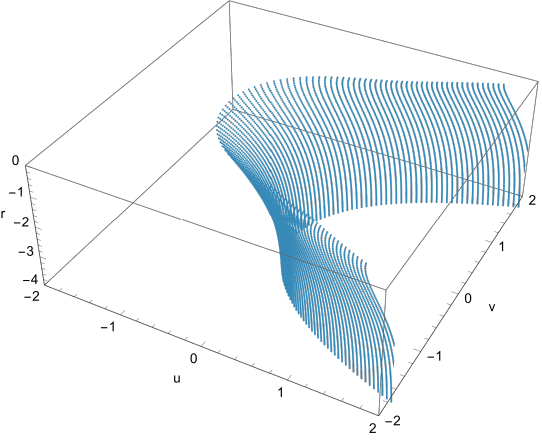}
  \caption{$t=1$.}
\end{subfigure}\\
\begin{subfigure}{0.47\textwidth}
  \centering
  \includegraphics[width=\textwidth]{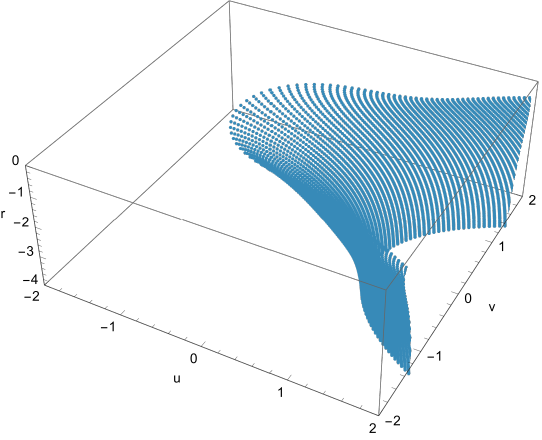}
  \caption{$t=2$.}
\end{subfigure}
\hfill
\begin{subfigure}{0.47\textwidth}
  \centering
  \includegraphics[width=\textwidth]{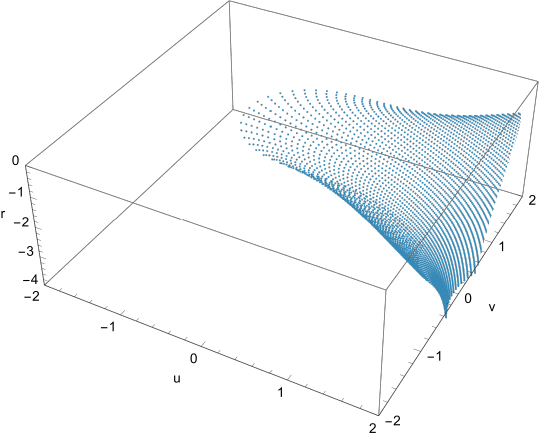}
  \caption{$t=3$.}
\end{subfigure}\\
\begin{subfigure}{0.47\textwidth}
  \centering
  \includegraphics[width=\textwidth]{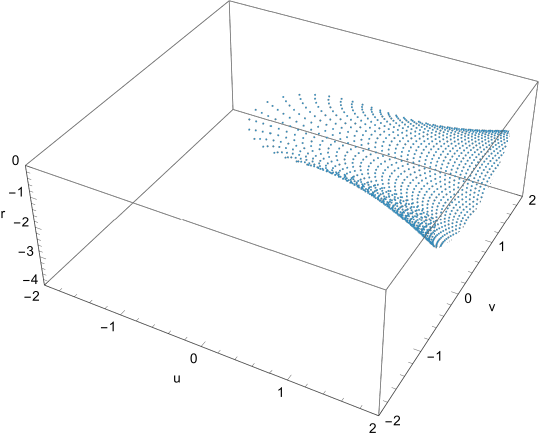}
  \caption{$t=4$.}
\end{subfigure}
\hfill
\begin{subfigure}{0.47\textwidth}
  \centering
  \includegraphics[width=\textwidth]{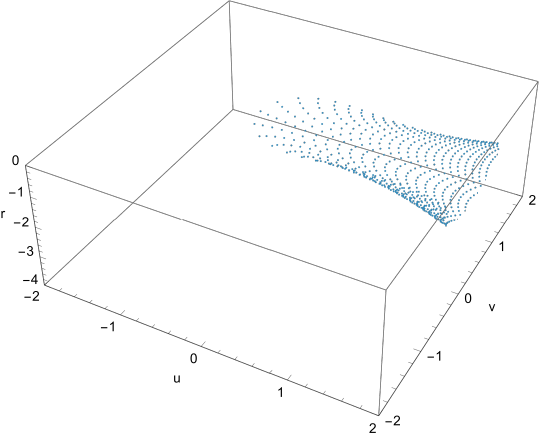}
  \caption{$t=5$.}
\end{subfigure}
\caption{Projection of the image of Lagrangian $U$ under the flow $\varphi_t$ generated by $X_H$ onto the $(u,v,r)$-coordinates.}
\end{figure}

\begin{figure}[H]
\centering
\begin{subfigure}{0.47\textwidth}
  \centering
  \includegraphics[width=\textwidth]{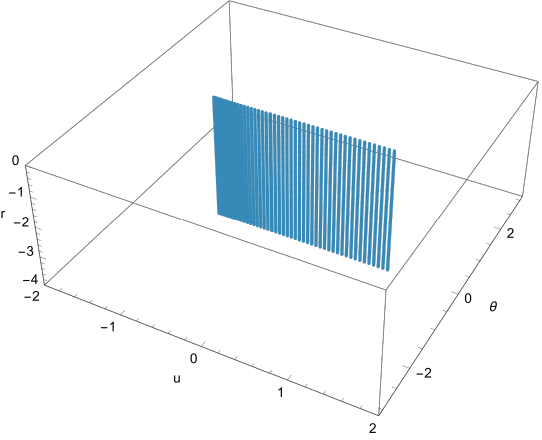}
  \caption{$t=0$.}
\end{subfigure}
\hfill
\begin{subfigure}{0.47\textwidth}
  \centering
  \includegraphics[width=\textwidth]{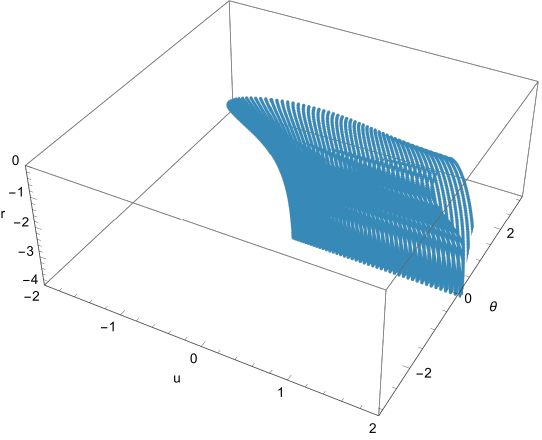}
  \caption{$t=1$.}
\end{subfigure}\\
\begin{subfigure}{0.47\textwidth}
  \centering
  \includegraphics[width=\textwidth]{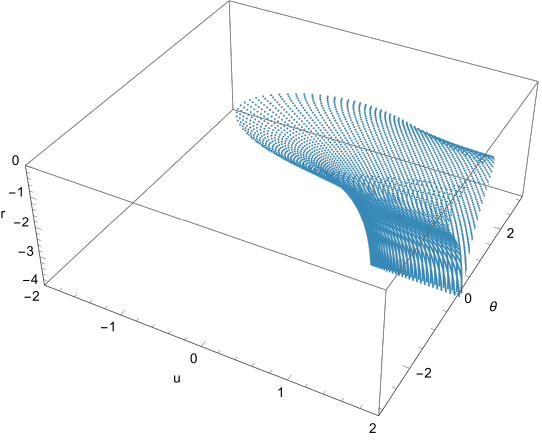}
  \caption{$t=2$.}
\end{subfigure}
\hfill
\begin{subfigure}{0.47\textwidth}
  \centering
  \includegraphics[width=\textwidth]{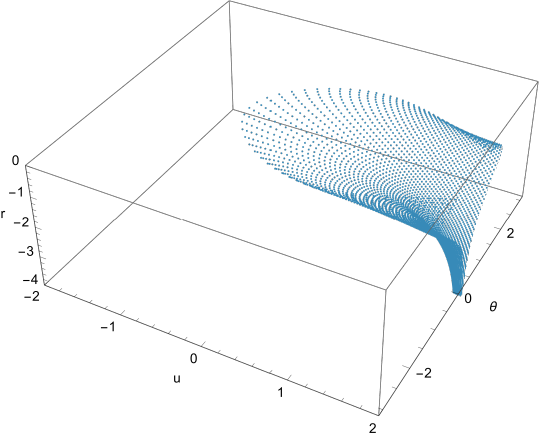}
  \caption{$t=3$.}
\end{subfigure}\\
\begin{subfigure}{0.47\textwidth}
  \centering
  \includegraphics[width=\textwidth]{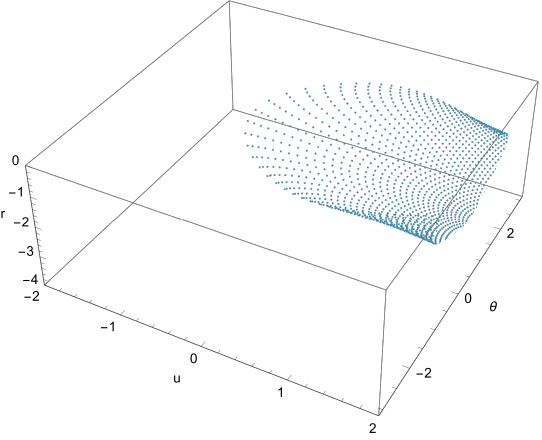}
  \caption{$t=4$.}
\end{subfigure}
\hfill
\begin{subfigure}{0.47\textwidth}
  \centering
  \includegraphics[width=\textwidth]{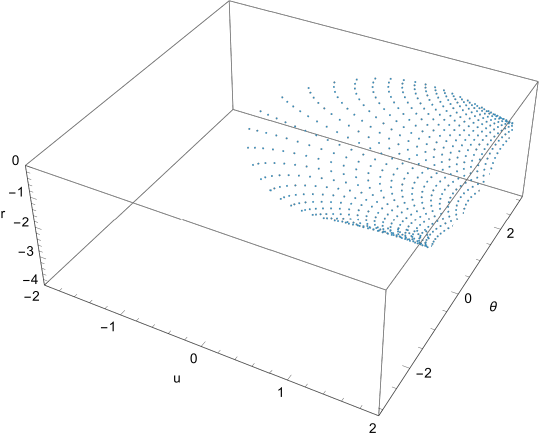}
  \caption{$t=5$.}
\end{subfigure}
\caption{Projection of the image of Lagrangian $U$ under the flow $\varphi_t$ onto the $(u,\theta,r)$-coordinates. Note that the boundaries $\theta=-\pi$ and $\theta=\pi$ of these cuboids are to be identified.}
\label{MMA2}
\end{figure}

\begin{figure}[H]
\centering
\begin{subfigure}{0.26\textwidth}
  \centering
  \includegraphics[width=\textwidth]{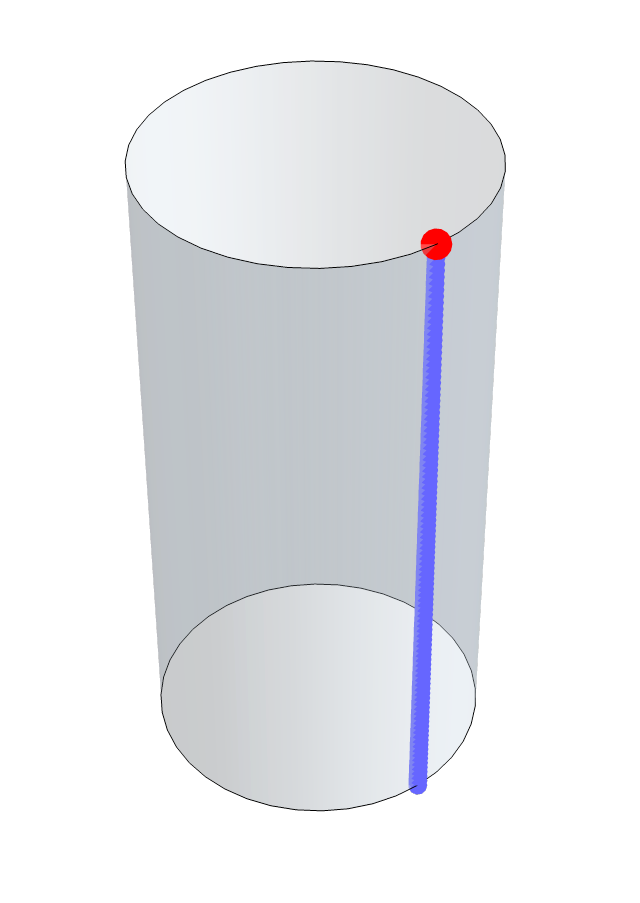}
  \caption{$t=0$.}
\end{subfigure}
\hfill
\begin{subfigure}{0.26\textwidth}
  \centering
  \includegraphics[width=\textwidth]{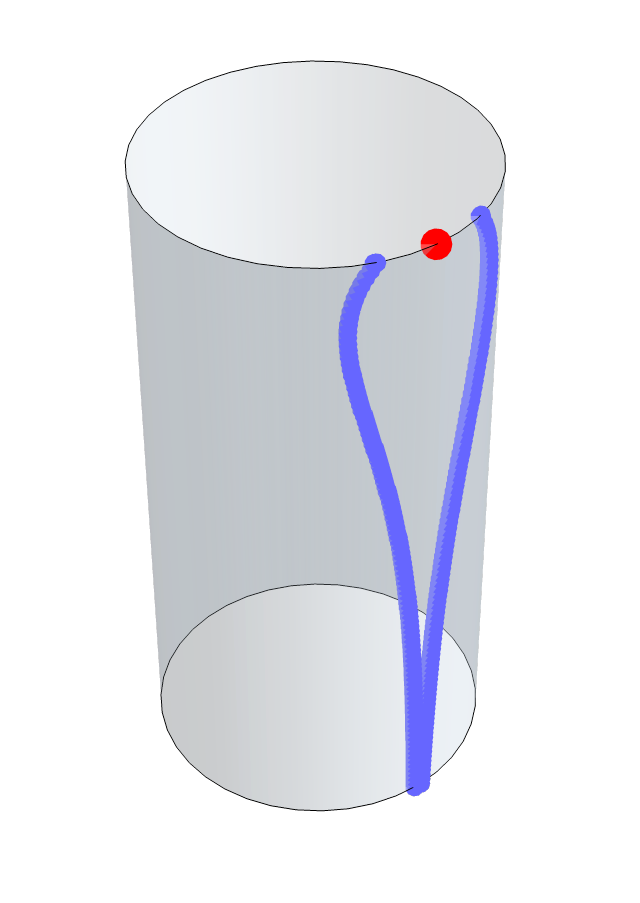}
  \caption{$t=1$.}
\end{subfigure}
\hfill
\begin{subfigure}{0.26\textwidth}
  \centering
  \includegraphics[width=\textwidth]{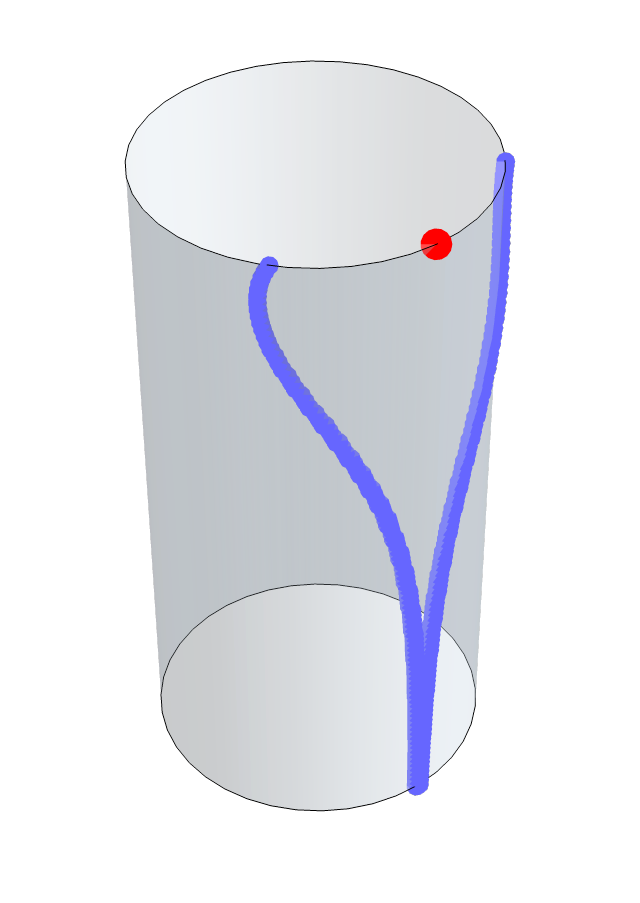}
  \caption{$t=2$.}
\end{subfigure}
\\
\begin{subfigure}{0.26\textwidth}
  \centering
  \includegraphics[width=\textwidth]{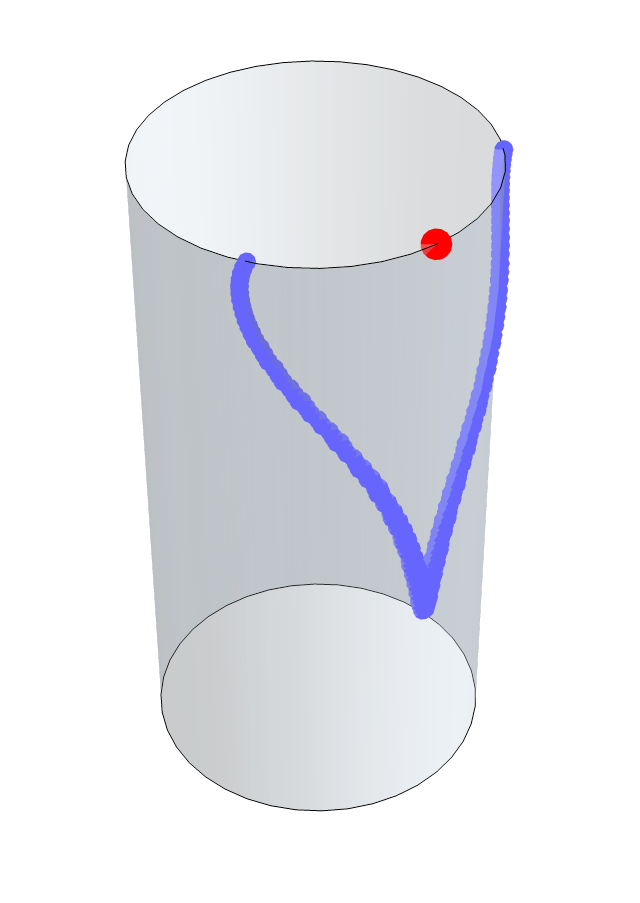}
  \caption{$t=2.2$.}
\end{subfigure}
\hfill
\begin{subfigure}{0.26\textwidth}
  \centering
  \includegraphics[width=\textwidth]{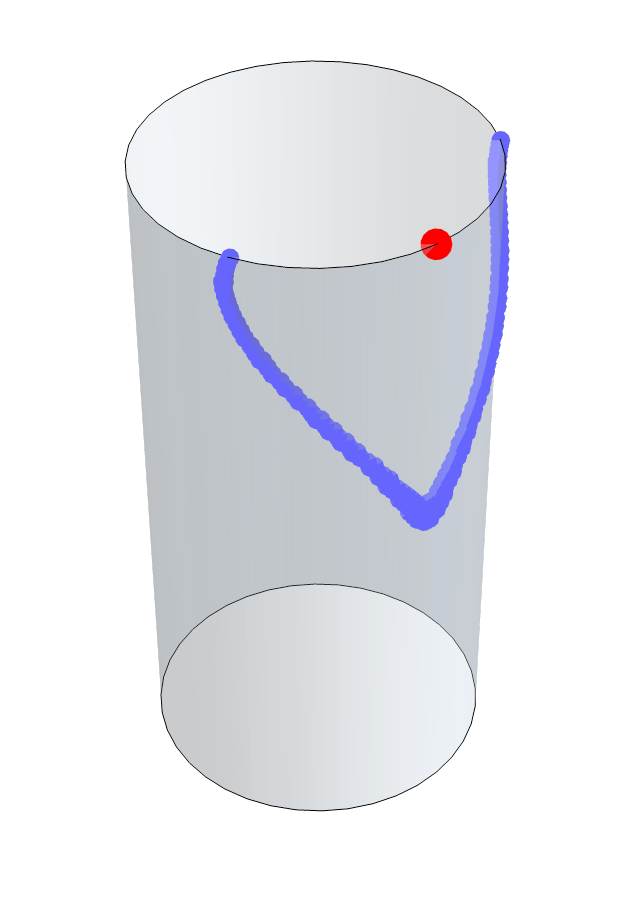}
  \caption{$t=2.4$.}
\end{subfigure}
\hfill
\begin{subfigure}{0.26\textwidth}
  \centering
  \includegraphics[width=\textwidth]{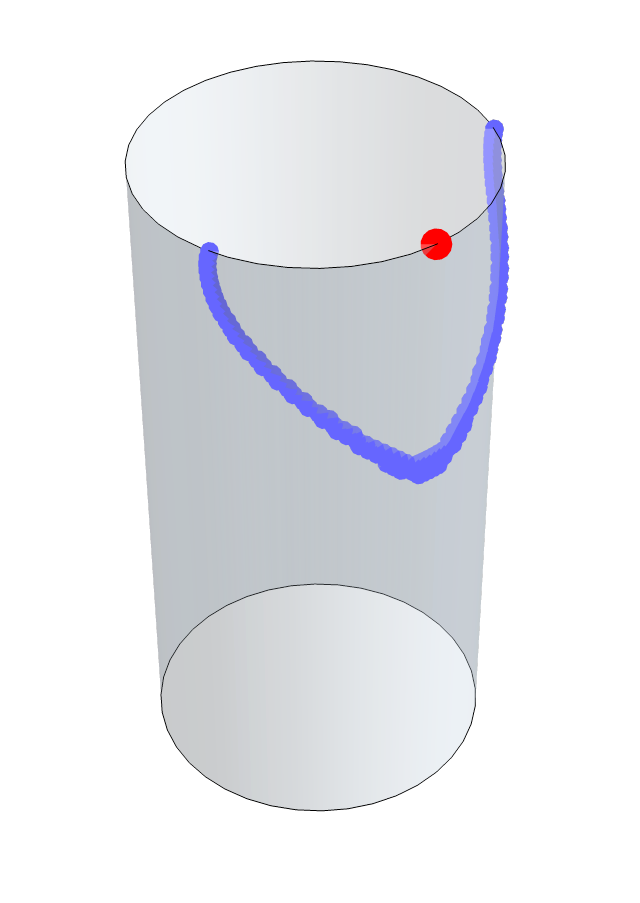}
  \caption{$t=2.6$.}
\end{subfigure}
\\
\begin{subfigure}{0.26\textwidth}
  \centering
  \includegraphics[width=\textwidth]{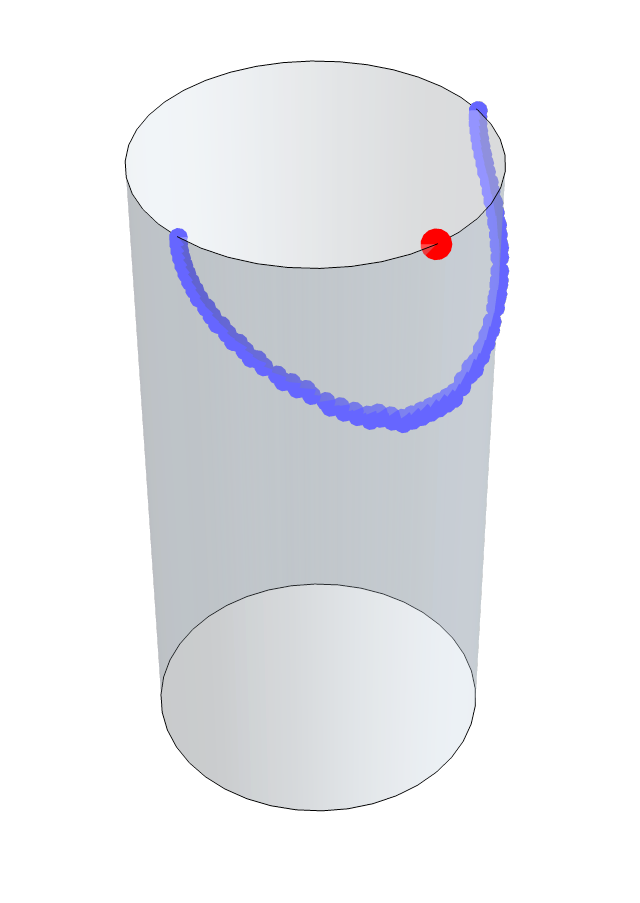}
  \caption{$t=3$.}
\end{subfigure}
\hfill
\begin{subfigure}{0.26\textwidth}
  \centering
  \includegraphics[width=\textwidth]{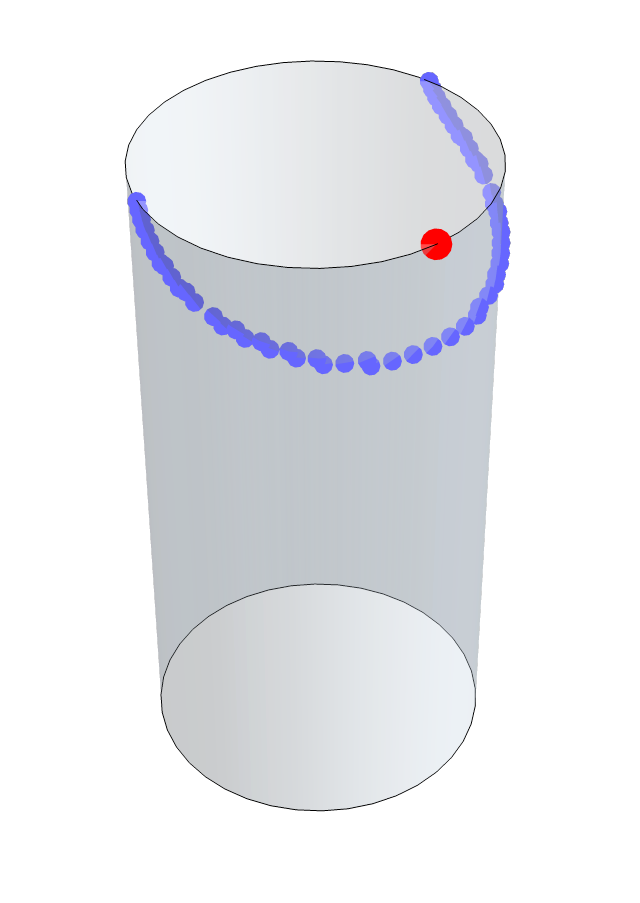}
  \caption{$t=4$.}
\end{subfigure}
\hfill
\begin{subfigure}{0.26\textwidth}
  \centering
  \includegraphics[width=\textwidth]{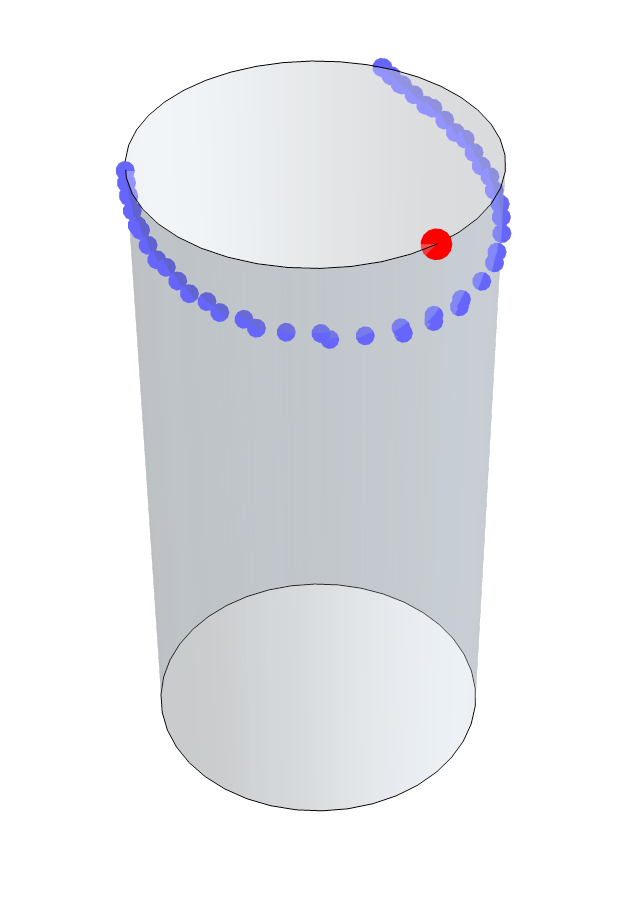}
  \caption{$t=5$.}
\end{subfigure}
\caption{The emergence of $I$-fiber. This is the projection of the points with $u=1$ in the image of Lagrangian $U$ under the flow $\varphi_t$ onto the fiber $\Cx_y$. The $v$-coordinate is suppressed. Note that the points lie in different fibers for finite $t$ and coincide only as $t \to \infty$.}

\label{MMA3}
\end{figure}

\bibliographystyle{plain}
\bibliography{ref}

\end{document}